%% file: Final-black.tex
\documentclass[final,onefignum,onetabnum]{siamart171218}

\input{ex_shared}
\input{tcilatex}
\ifpdf
\hypersetup{
  pdftitle={Convergence and Semi-convergence of a class of constrained
block iterative methods},
  pdfauthor={M. Mirzapour, Andrzej Cegielski and T. Elfving}
}
\fi

\begin{document}

\maketitle


\begin{abstract}
In this paper, we analyze the convergence 
properties of projected non-stationary block iterative methods (P-BIM)
aiming to find a constrained solution to large linear, usually both noisy and
ill-conditioned, systems of equations.
We split
the error of the $k$th iterate into noise error and iteration error, and
 consider each error separately.
The iteration error is treated for a more general algorithm, also
suited for solving
split feasibility problems in Hilbert space. The results for P-BIM come
out as a special case.
The algorithmic step involves projecting onto closed convex sets.
When these sets are
polyhedral, and of finite dimension, it is shown that the algorithm
converges linearly.
We further
derive  an upper bound for the noise error of P-BIM. Based on this
bound, we suggest a new strategy for choosing relaxation parameters,
 which assist in speeding up
the reconstruction process and improving the quality of obtained images.
The relaxation parameters may depend on the noise.
The performance of the suggested  strategy is shown by examples taken
from the field of image reconstruction from projections.
\end{abstract}




\section{Introduction}

Among reconstruction methods in computerized tomography (CT) iterative methods
are particularly suited for their ability to incorporate constraints
(e.g. nonnegativity) on the sought solution. This is important in
situations with few and/or noisy data, as in limited data CT when
exposition to a low dose of X-rays is an issue.
After discretization of the underlying integral transform a large, sparse,
unstructured and ill-posed (sensitive to noise) linear system occurs:
\begin{equation}\label{sys:equat}
Ax=b,
\end{equation}
where $A\in \mathbb{R}^{m\times n}$ and $b\in \mathbb{R}^{m}$ are known,
 $x\in \mathbb{R}^{n}$ is the unknown image to be estimated. We will, in connection with each algorithm/problem
discuss possible consistency assumptions on the rhs $b$.


\subsection{Some background}
The first reported use of iteration in CT seems to be the paper \cite{gordon1970algebraic}. There,
in its basic form, the given method turned out to be identical to the Kaczmarz
method \cite{kaczmarz1937}. Then for each equation in (\ref{sys:equat}) the new iterate is constructed
such that the previous equation is satisfied. This is done for each of the
$m$ equations in turn, and then the process is repeated. For more on the history see \cite{gordon2010stop}. Somewhat later
in \cite{gilbert1972iterative} a fully simultaneous method, SIRT, was proposed where the new
iterate is constructed by adding to the previous one the gradient of the least
squares functional corresponding to (\ref{sys:equat}). Later block-iterative versions were
developed and analyzed, where, instead of using a single row, blocks of rows
are used in each iteration. A very general form of block-iteration is analyzed
in \cite{aharoni1989block}. For more on the development of block-iteration see, e.g. the recent
survey \cite{elfving2018row}, and further in Subsection \ref{BIM-section}.
To cope with constraints it was, early on, suggested to project the iterate
onto the nonnegative orthant, either after each iterative step or after a whole
cycle, i.e. after completing a full sweep through the matrix $A$. The order
in which the rows/blocks are used is called control, and can have quite an
impact of the behavior of the method \cite{herman2009fundamentals}. In cyclic control the rows/blocks
are picked up in their original order in the matrix. For some other controls,
e.g. almost cyclic control see \cite{C2013}. A recent strategy is stochastic control
which has attracted much interest see, e.g. \cite{Ne2019,NP2014,RPT2020}. We will however not
consider this control in our paper.
\subsection{Simultaneous Algorithms}\label{sec1.1}
Let $M$ and $N$ be given symmetric positive definite (SPD) matrices,
 and denote by $Q^{1/2}$  the square root of an SPD matrix $Q$.
Moreover $\{\lambda_{k}\}_{k=0}^{\infty}$ is a sequence of
  positive relaxation parameters, and  $\tilde{A}=M^{1/2}AN^{1/2}$.
The simultaneous iterative reconstruction technique (SIRT) is a class of iterative
and gradient-based methods defined as
\begin{equation}\label{general:landweber}
x^{k+1}=x^{k}+
\frac{\lambda_{k}}{\|\tilde{A}\|^{2}}
NA^{T}M\left(b-Ax^{k}\right)~~~~k=0,1,2,\ldots,
\end{equation}
see e.g. \cite{HANSEN2018air}. Here $\|\tilde A\|=\|\tilde A\|_2$ (note that
$\|\tilde A\|^2=\|\tilde A^T\tilde A\|$).
By premultiplying (\ref{general:landweber}) by $N^{-1/2}$ and putting $N^{-1/2}x^k=y^k$
one gets the equivalent iteration
\begin{equation*}
y^{k+1}=y^{k}+\frac{\lambda_k}{\|\tilde A\|^2}\tilde A^T(M^{1/2}b-\tilde Ay^k),
\end{equation*}
which as is well known converges for $\lambda_k \in [\varepsilon, 2-\varepsilon]$, where $\varepsilon$ is small positive constant. For the stationary
case, $\lambda_k=\lambda$ this is also a necessary condition.
It follows that the sequence $\{x^k\}$
converges towards an $M-$weighted least squares solution of (\ref{sys:equat}) provided
$\lambda_k \in [\varepsilon, 2-\varepsilon]$. Hence  no consistency assumption is needed
on $b$. For a more general result (for the nonstationary case) see \cite{Qu2008}.
Both $M$ and
$N$ can have  an efficacious impact on the initial speed of convergence.
A computational analysis of SIRT applied to CT-problems appears in \cite{Gregor2008}.

Let diag$(B)$ be a diagonal matrix with
$B_j$ being the $j$th diagonal element, and let Id denote
 the identity matrix of appropriate dimension (or in the case of infinite
dimensional Hilbert spaces the identity operator).
Further $a^{i}$ is $i$th row of
 $A$, $||.||$ is the Euclidean 2-norm, and $||.||_W$ is the weighted Euclidean
norm. Here $W$ is a given SPD matrix.

Some well-known choices
for
 $M$ and $N$ yielding fully simultaneous iterations
 are listed below:

\begin{itemize}
  \item $N=M=$Id, lead to the Landweber method
\cite{bertero1998introduction,engl1996regularization}.
  \item $N=$Id and $M=D_{C}=\frac{1}{m} diag\left(\frac{1}{\|a^{i}\|^2}\right)$
  lead to the Cimmino method \cite{cimmino1938}.
    \item $N=$Id and
$M=D_{W}=\frac{1}{m} diag\left(\frac{1}{\|a^{i}\|_{W}^2}\right)$ lead to CAV
(Component Averaging Method) \cite{censor2001component}.
     Here $W=diag(w_j)$  where $w_j$
equals  the number
of nonzero elements in the  $j$th column of $A$.
  \item $N=W^{-1}$ and $M=mD_C$ lead to DROP
(Diagonally Relaxed Orthogonal Projection)
   method \cite{censor2008diagonally}.
  \item $N=diag(\mbox{row sums})^{-1}$ and $M=diag (\mbox{column  sums})^{-1}$
lead
   to SART (Simultaneous Algebraic Reconstruction Technique)
\cite{avinash2001principles}.
  \item $N=$Id and $M=(2-\lambda)(D+\lambda L^{T})D(D+\lambda L)^{-1}$ where
   $L$ is the left triangular part of $AA^T$, $D$ its diagonal and taking
$\lambda_k=\lambda$  fixed, lead to
the symmetric Kaczmarz's method \cite{EN2009}.
\end{itemize}
 \subsection{Projected SIRT}
Consider the linear system of equations (\ref{sys:equat}) subject to $x\in C$ where $C$ denotes a closed convex set
and $P_C$
the operator of metric projection onto the set $C$.
 Inserting a  projection onto a convex set 
  after each iteration  can reduce the error and improve the quality of the
reconstructed image, see, e.g.
\cite{NM2015,piana1997projected} and Figure \ref{cons} below.
 We will now assume that $N=\textrm{Id}$ (but remark on the case $N\neq \textrm{Id}$
at the end of Section \ref{Conv:Analysis}).
Put
\begin{equation}
\hat{A}=M^{1/2}A, \ \hat{b}=M^{1/2}b.
\end{equation}
Then the
P(Projected)-SIRT Algorithm is (note that $\|\hat{A}\|^2=\|A^TMA\|$)
\begin{algorithm}
\caption{P-SIRT Algorithm}
\label{pro-algt1}
\begin{algorithmic}[1]
\State $x^{0}\in \mathbb{R}^n$ is an arbitrary starting point.

\For{$k\geq 0$ }
        \vspace{0.2cm}
%
$x^{k+1}=U(x^k)=P_{C}\left(x^{k}+
 \frac{\lambda_{k}}{\|\hat{A}\|^2}
\hat{A}^{T}\left(\hat{b}-\hat{A}x^{k}\right)\right)$
                \vspace{0.2cm}
        \EndFor
\end{algorithmic}
\end{algorithm}

Let $\limfunc{Fix}\hat{U}$ denote the set of fixed points of an operator
$\hat{U} \ : \ \mathbb{R}^n \rightarrow \mathbb{R}^n$.
The following characterization (adapted to our problem formulation) is shown
in \cite[Proposition 2.1]{Byr02} (also given in
\cite[Proposition 4.7.2]{C2013}),
\begin{equation} \label{fixP}
\limfunc{Fix}U=\arg\min_{x\in C}\|b-Ax\|^2_M.
\end{equation}

We have the following convergence result for Algorithm \ref{pro-algt1}.

\begin{theorem}\label{Theorem:1}

Let $\{x^{k}\}_{k=0}^{\infty}$ be the sequence  generated by Algorithm
 \ref{pro-algt1}, and let for $k\geq 0, \
\lambda_{k}\in (\varepsilon, 2-\varepsilon)$ for some small $\varepsilon > 0$.
%
Then for an
 arbitrary $x^{0}\in \mathbb{R}^n$ the sequence
$\{x^{k}\}_{k=0}^{\infty}$ converges to a fixed point of $U$
provided such a fixed point exists.
\end{theorem}
\begin{proof} For $\lambda_k=\lambda$ see \cite[Theorem 2.1]{Byr02}.
For the non-stationary case see \cite{Mou10, NM2015}.
\end{proof}
Note that the same interval for the relaxation parameters is valid both for
SIRT and P-SIRT.
\begin{remark}
When $ker(A)=\emptyset$ then $U$ has a unique fixed point for any $C$.
When $ker(A)$ is not empty, e.g. when $A$ is underdetermined, then $U$ has
a fixed point when $C$ is compact, e.g. $C$ equals the unit hypercube (as in our
experiments). One may also note that with $C$ the rangespace of $A^T$ the unique
fixed point of $U$ is the pseudo-inverse solution.
\end{remark}




\subsection{Block-iterative Algorithms}\label{BIM-section}
There are two issues related to the
practical use of SIRT-methods on large data problems. One is the need for computer memory,
and the other the often slow rate of initial rate of convergence.
By partitioning the matrix into row-blocks and iterating sequentially
over the blocks one can improve on both. A careful study of implementation
of block-iterative methods on a modern multicore platform is presented in \cite{sorensen2014multicore}.
We mention also the possibility of performing the iterations in parallel
over the blocks \cite{sorensen2014multicore,NA2015}. However, here we do not consider this case.
The use of block-iteration has been quite successful in CT-applications,
where it is often referred to as Ordered Subset Iteration \cite{byrne1996block,Hudson1994}.
A thorough convergence analysis of block-iteration used in CT
is \cite{JW2003}.  A Landweber-Kaczmarz type block-iteration is proposed
and analyzed in \cite{haltmeier2009}.
Block-iterative versions of the examples listed in
Section \ref{sec1.1} are also presented in the corresponding references.

Let the matrix $A$
 be partitioned into $p$ blocks of
equations, which may contain common equations but each equation
should appear  at least in one of the blocks.
Choose $A_{t}\textrm{ and } b^{t}$ as the $t$th row-block of $A$ and
$b,$ respectively.
Let
$\{M_{t}\}_{t=1}^p$   be a set of given SPD
weight matrices.
Further $\{C_t\}_{t=1}^p$ is a {family} of closed convex sets.
Then \eqref{sys:equat} subject to $x\in C$ is equivalent to the following block problems:
\begin{equation}
\text{find }x\in C_{t}\text{ satisfying }A_{t}x=b^{t}\text{, }t=1,2,...,p%
\text{.}  \label{e-xinCt}
\end{equation}
Defining $C=\bigcap_{t=1}^{p} C_{t}$ problem \eqref{e-xinCt} can be also written as find $x\in C$ satisfying $Ax=b$. In our study, we prefer, however, form \eqref{e-xinCt}, because in algorithms for solving \eqref{e-xinCt} we need only a sequential access to the data. Recently, a more general form of this problem has been defined in \cite[ Eq. (1.3)]{reich2020split}, see also \cite[Algorithm 1]{yukawa2010multi}.

Define cyclic control of the data by
\begin{equation} \label{cycc}
i_{k}=[k]:=k(\func{mod}p)+1.
\end{equation}
Put
\begin{equation}
\hat{A}_{[k]}=M_{[k]}^{1/2}A_{[k]}, \
\hat{b}^{[k]}=M_{[k]}^{1/2}b^{[k]}.
\end{equation}
%

With these notations the
Projected Block-Iterative Method (P-BIM) is
\begin{algorithm}
\caption{Projected Block Iterative Method (P-BIM)}
\label{BIP:algt}
\begin{algorithmic}[1]
\State $x^{0}\in \mathbb{R}^n$ is an arbitrary starting point
   \For{$k\geq 0$}
$x^{k+1}=U_k(x^{k})=
P_{C_{[k]}}\left (x^{k}+\frac{\lambda_k}{\|\hat{A}_{[k]}\|^2}\hat{A}_{[k]}^T
\left(\hat{b}^{[k]}-\hat{A}_{[k]}x^{k}\right )\right ) $
\EndFor
 \end{algorithmic}
\end{algorithm}

Note that the $k-$dependence is in the relaxation parameters, and the cyclic
access of the data $\{C_t,A_t,M_t,b^{t}\}, \ t=1,2,\dots,p$.
A special case of the above algorithm has been studied in \cite[Eq. (7)]{NK2015}, where  $C_t=\mathbb{R}^n$, $t=1,2, \cdots ,p$.
A parallel gradient based projection algorithm has been studied in \cite[Algorithm 4]{CZ2014} and \cite{reich2020split}. The difference with respect to our algorithm is that they use the full gradient, whereas in our approach only a part of the gradient (of the underlying least squares functional) is used, which gives rise to the block structure. In the optimization literature such methods are sometimes called
incremental methods. The method $(1.4)$ and $(1.5)$ in \cite{nedic2001incremental} with $f=\|Ax-b\|^{2}_{M}$,
and assuming the relaxation parameter is constant during a whole cycle is identical
to P-BIM. In \cite{nedic2001incremental} several relaxation strategies are investigated. The one most similar
to our choice is using a constant parameter $\lambda_{k}=\lambda$. Here however, the convergence
result is rather weak see \cite[Proposition 2.1]{nedic2001incremental}.

We next characterize the fixed points of the operator $U_k$ in P-BIM.
Similarly as in (\ref{fixP}) we get
\begin{equation}
\limfunc{Fix}U_k=\arg\min_{x\in C_{[k]}}\|b^{[k]}-A_{[k]}x\|^2_{M_{[k]}}.
\end{equation}
Put
$$
\hat{U}_t(x)=
P_{C_{t}}\left (x+\frac{\lambda}{\|\hat{A}_{t}\|^2}\hat{A}_{t}^T
\left(\hat{b}^{t}-\hat{A}_{t}x\right )\right ),~t=1,2,\dots,p.
$$
Then
$$
\limfunc{Fix}\hat{U}_t=
\arg\min_{x\in C_{t}}\|b^{t}-A_{t}x\|^2_{M_{t}}, \
t=1,2,\dots,p.
$$
It holds
$$
\bigcap_{k=0}^{\infty}\limfunc{Fix}U_{k}=
\bigcap_{t=1}^{p}\limfunc{Fix}\hat{U}_{t}.
$$

The following result will be shown in
 Section \ref{Conv:Analysis} in Hilbert space settings.
\begin{theorem}\label{main:thm}
Let $\varepsilon\in (0,1)$. Assume that $\lambda_k\in[\varepsilon, 2-\varepsilon]$.
If $\bigcap_{t=1}^{p}\limfunc{Fix}\hat{U}_{t}\neq\emptyset$,
then the sequence $\{x^k\}_{k=0}^{\infty}$ produced by the Algorithm P-BIM
converges to $ x^{*}\in \bigcap_{t=1}^{p}\limfunc{Fix}\hat{U}_{t}$. Further if $\{C_t\}$
is polyhedral the iteration converges linearly.
\end{theorem}


\subsection{Noise-error}
A common situation is when  the right-hand side
$b$ is corrupted by noise, i.e. $b=\bar{b}+\delta b$,
where $\bar{b}$ denotes the exact right-hand side.
Let $\bar{x}^k$ be the iteration vector
using the exact right-hand side $\bar{b}$ and let $x^*$
be a fixed point of the unperturbed iteration.
The total error can be decomposed into two terms
  \begin{equation}
 x^k-x^* = (x^k-\bar{x}^k) + (\bar{x}^k-x^*).
  \end{equation}
The first term is called the \emph{noise error},
and the second the \emph{iteration error}.
During the first iterations of a convergent method the iteration error
dominates,
and hence the total error decreases -- but after a while the noise
error starts to
grow resulting in so-called semi-convergence, as will be demonstrated in
 Section \ref{Numerical}.
This   phenomenon originally defined in \cite[Section IV.1]{natterer1986mathematics} is further studied  in
\cite{natterer1986mathematics,engl1996regularization,bertero1998introduction,
herman2009fundamentals}.
%
It has been observed experimentally that the minimum error, i.e.
$\min_{k\geq 0}\|x^{k}-x^{*}\|$,
  is almost independent of both relaxation parameters and the weight matrix.
However the rate of decrease of the iteration error and the rate of
increase of the noise error are affected by both.
 The noise-error measures the deviation due to errors in the data. Here we will only consider errors in the rhs $b$. Possible errors in $A$ are usually caused by modelling errors, and should then be dealt with using model refinements. We will study how the noise-error evolves as a function of the iteration index $k$. Rather trivially it is proportional to $k$. However we will derive a bound proportional to $\sqrt{k}$, see \eqref{UpperNew}. We next describe some earlier contributions.
In \cite{elfving2010}, an upper bound for the
noise error of SIRT with a fixed relaxation parameter was derived.
Based on  minimizing
this  upper bound, two strategies
for choosing relaxation parameters
were suggested. The obtained sequences of
 relaxation parameters are both non-negative and non-ascending.
Later in \cite{EHN2012}
these   strategies were analyzed and used in P-SIRT. 
  An upper bound for the noise error of the block version of SIRT was derived
in \cite{NK2015}, and  the above two strategies for picking relaxation
parameters together with a third choice were studied.
Recently, the stationary case of SIRT and P-SIRT
were studied in \cite{Elfving2014}, \cite{KL-14}
when the M-matrix is non-symmetric.
This case includes the classical Kaczmarz method \cite{kaczmarz1937}.


\subsection{Organization}\label{motivations}
In Section \ref{Conv:Analysis} we  define and study a
sequential block-iterative iteration (of which P-BIM is a special case)
 adapted for solving split feasibility problems in Hilbert space.
The proof of Theorem \ref{main:thm} then follows from a more general result
in Section \ref{Conv:Analysis}. We further study the noise error of
P-BIM in Section \ref{SIRT:Sec}, and derive  a new upper  bound.
In particular we extend the noise error analysis of block-iteration
to {\it projected} block-iteration, and also to the case
 when the relaxation parameters are allowed to depend  on the noise.
In Section \ref{relaxation:par} we propose a new rule
for picking relaxation parameters, and present its properties.
We demonstrate the performance  of P-BIM using this and
other relaxation strategies with examples taken from
tomographic imaging in Section \ref{Numerical}.


\section{Hilbert space analysis}\label{Conv:Analysis}

In this section we will first formulate a quite general feasibility problem.
Then we continue to study algorithms, and their convergence properties for
the problem. It will be shown that P-BIM comes out as a special case.

Let $\mathcal{H}$ be a real Hilbert and $C_{i}\subseteq \mathcal{H}%
,i=1,2,...,m,$ be closed convex subsets. The \textit{convex feasibility
problem} (CFP) is to find $x\in \bigcap_{i=1}^{m}C_{i}$ if such a point
exists. CFP is a classical problem and was studied within the last decades by
many researchers, see, e.g. \cite{BB96, C2013, BC17} and the references
therein, where several methods for solving the CFP and their convergence
properties were presented. First we recall some notions regarding operators
in a Hilbert space. Let $X\subseteq \mathcal{H}$ be a nonempty convex set
and let $T:X\rightarrow \mathcal{H}$ be an operator with a fixed point. Then
$T$ is called $\alpha $-\textit{strongly quasi-nonexpansive} ($\alpha -$%
SQNE), where $\alpha \geq 0$, if
\[
\Vert Tx-z\Vert ^{2}\leq \Vert x-z\Vert ^{2}-\alpha \Vert Tx-x\Vert ^{2},\
\forall x\in X,\text{ and }z\in \limfunc{Fix}T,
\]%
\cite[Definition 2.1.38]{C2013}. The operator $T$ is called a \textit{cutter}
\cite[Definition 2.1.30]{C2013} if for all $x\in X$ and for all $z\in
\limfunc{Fix}T$ it holds%
\[
\langle x-Tx,z-Tx\rangle \leq 0\text{.}
\]%
By \cite[Section 11]{Goebel1984}, a firmly nonexpansive operator having a fixed point is a cutter, see e.g., \cite[Section 2.2]{C2013} or \cite[Section 4.1]{BC17} for a definition of a firmly nonexpansive operator.

Further, $T$ is called \textit{weakly regular} if $\limfunc{Id}-T$ is demi-closed at $%
0$, i.e., $x_{k}\rightharpoonup x$ and $x^{k}-Tx_{k}\rightarrow 0$ yields $%
x\in \limfunc{Fix}T$. Here $\rightharpoonup $ is a symbol for weak
convergence. $T$ is called \textit{linearly regular} if there is a constant $%
\delta >0$ such that for all $x\in X$ it holds
\[
d(x,\limfunc{Fix}T)\leq \delta \Vert Tx-x\Vert \text{,}
\]%
cf. \cite[Definition 3.1]{CRZ18}, where $d(x,C):=\inf_{y\in C}\Vert x-y\Vert
$ denotes the distance of $x$ to a subset $C$. In a similar way one can
define weak and linear regularity of sequences of operators, see \cite[%
Definition 4.1]{CRZ18} for details.

\subsection{Split feasibility problem}

Let $\mathcal{H},\mathcal{H}_{0}$ be real Hilbert spaces, $C\subseteq
\mathcal{H}$ and $Q\subseteq \mathcal{H}_{0}$ be closed convex subsets and $%
A:\mathcal{H}\rightarrow \mathcal{H}_{0}$ be a nonzero bounded linear
operator. The \textit{split feasibility problem }(SFP) is to
\[
\text{find }x\in C\text{ such that }Ax\in Q\text{.}
\]%
A more general form of the SFP is to
\begin{equation}\label{F1}
\text{minimize}~ \Vert P_{Q}(Ax)-Ax\Vert ^{2}~ \text{subject to}~ x\in C,
\end{equation}
or, equivalently, to find a fixed point of the operator
$T:\mathcal{H}\rightarrow \mathcal{H}$ defined by $T(x)=P_{C}(x+\frac{1}{%
\Vert A\Vert ^{2}}A^{\ast }(P_{Q}(Ax)-Ax)$. The SFP was introduced by Censor
and Elfving \cite{CE-1995} in the finite dimensional settings. Byrne \cite%
{Byr02} proved the weak convergence of the so called $CQ$-method, $%
x^{k+1}=T_{\lambda }(x^{k})$, to $\limfunc{argmin}_{x\in C}\Vert
P_{Q}(Ax)-Ax\Vert $ which is a solution of the SFP. Here $\lambda \in (0,2)$
is a relaxation parameter and $\Vert A\Vert $ denotes the spectral norm of $A
$. In the last $20$ years the SFP and its variants was intensively studied by
many researchers, see, e.g.  \cite{Mou11, Xu10, WX11, Ceg15, CRZ20,Reich2021} to
mention some of them.

\subsection{Block SFP}

Let $\mathcal{H},\mathcal{H}_{i}$, $i\in I:=\{1,2,\dots ,p\}$ be Hilbert
spaces. Consider a family of problems:
\begin{equation}
\text{find }x\in C_{i}\text{ such that }A_{i}x\in Q_{i}\text{,}  \label{p1}
\end{equation}%
$i\in I$, where $C_{i}\subseteq \mathcal{H},Q_{i}\subseteq \mathcal{H}_{i}$
are nonempty closed convex subsets, and $A_{i}:\mathcal{H}\rightarrow
\mathcal{H}_{i}$ are nonzero bounded linear operators, $i\in I$. We call
this family of problems, \textit{the block SFP}. Denote by
\begin{equation}
F_{i}:=C_{i}\cap A_{i}^{-1}(Q_{i})\text{,}
\end{equation}%
the solution set of problem (\ref{p1}), $i\in I$. Define
\[
\mathcal{H}_{0}:=\mathcal{H}_{1}\times \mathcal{H}_{2}\times \dots \times
\mathcal{H}_{p},
\]%
and a bounded linear operator $A:\mathcal{H}\rightarrow \mathcal{H}_{0}$ by
\begin{equation}
Ax=(A_{1}x,A_{2}x,\dots ,A_{p}x).
\end{equation}%
Put
\begin{equation}
Q=Q_{1}\times Q_{2}\times ,\dots \times Q_{p},\ C=C_{1}\cap C_{2}\cap \dots
\cap C_{p}\text{,}
\end{equation}%
and suppose that there is a common solution of all problems (\ref{p1}), $%
i\in I$, i.e.,
\begin{equation}
\Omega :=\bigcap_{i=1}^{p}F_{i}=C\cap A^{-1}(Q)\neq \emptyset .
\label{e-omega}
\end{equation}%
The family of problems (\ref{p1}), $i\in I$, can be written as the following
split feasibility problem \cite{Byr02,C2013,CE-2005}
\begin{equation}\label{sfp}
\text{find }x\in C\text{ such that }Ax\in Q\text{.}
\end{equation}

This problem formulation includes problem (\ref{e-xinCt}) by setting $%
Q_{i}=\{b^{i}\}$, $i\in I$. By picking $Q=\{b\}$ (and using the M-norm) problem \eqref{F1} becomes
minimization of $\Vert Ax-b\Vert _{M}$ subject to $x\in C$, i.e. problem \eqref{fixP}.
Similarly when $p>1$ we retrieve the corresponding block-problem as expressed in Theorem
\ref{main:thm}. In fact one can reformulate the block SFP into an optimization problem
(as for the case $p=1$, and problem \eqref{F1}) see \cite{CE-2005} for details.
Several extensions of \eqref{sfp} have been proposed, \cite{CS09, BCGR12}.
In, e.g. \cite{CS09} the sets $C,Q$ were extended to fixed points sets of
certain operators, $S$ and $T$. With $S=P_{C}$, $T=P_{Q}$ the original
formulation is retrieved. An interesting extension is compressed sensing,
where a sparse solution of a consistent linear system is sought. Here one
uses $l_{1}$-regularization which yields the $Q$-lasso problem $\min_{x\in
C}\Vert \left(I-P_{Q}\right)Ax\Vert^{2}_{2}+\gamma \Vert x\Vert _{1}$. With $C=%
\mathbb{R}
^{n}$ and $Q=\{b\}$ the original lasso problem \cite{Tib96} is recovered.
For some further work along these lines see \cite{MG18}.

\bigskip

In order to prove Theorem \ref{main:thm}, we present below some general convergence
results which follow from \cite{CRZ18}. Let $S_{i}:\mathcal{H}%
_{i}\rightarrow \mathcal{H}_{i}$ be a cutter with $\limfunc{Fix}S_{i}=Q_{i}$%
, $i\in I$. Define an operator $T_{i}:\mathcal{H}\rightarrow \mathcal{H}$ by
\begin{equation}
T_{i}(x)=\mathcal{L}\{S_{i}\}(x):=x+\frac{1}{\Vert A_{i}\Vert ^{2}}%
A_{i}^{\ast }(S_{i}(A_{i}x)-A_{i}x)\text{,}
\end{equation}%
i.e., $T_{i}$ is the Landweber operator corresponding to $S_{i}$, $i\in I$
(cf. \cite[Definition 4.1]{CRZ20}). Then the $\lambda $-relaxation of $T_{i}$
is $T_{i,\lambda }:=(1-\lambda )\limfunc{Id}+\lambda T_{i}$. Hence
\begin{equation}
T_{i,\lambda }(x)=\mathcal{L}\{S_{i,\lambda }\}(x):=x+\frac{\lambda }{\Vert
A_{i}\Vert ^{2}}A_{i}^{\ast }(S_{i}(A_{i}x)-A_{i}x)\text{,}  \label{e-Til}
\end{equation}%
where $\lambda \in (0,2)$, $i\in I$. By \cite[Lemma 3.1]{WX11} and by (\ref%
{e-omega}), $T_{i}$ is a cutter and $\limfunc{Fix}T_{i,\lambda }=\limfunc{Fix%
}T_{i}=A_{i}^{-1}(Q_{i})\neq \emptyset $. Let $U_{i}$ be a cutter with $%
\limfunc{Fix}U_{i}=C_{i}$, $i\in I$, and $U_{i,\mu }$ be the $\mu $%
-relaxation of $U_{i}.$ Define
\begin{equation}
V_{i}^{\lambda ,\mu }:=U_{i,\mu }T_{i,\lambda }\text{,}
\end{equation}%
$i\in I$, where $\lambda ,\mu \in (0,2)$. Let $\{i_{k}\}_{k=0}^{\infty
}\subseteq \{1,2,...,p\}$ be a control sequence and $x^{k}$ be generated by
the iteration
\begin{equation}
x^{k+1}=U_{i_{k},\mu _{k}}(x^{k}+\frac{\lambda _{k}}{\Vert A_{i_{k}}\Vert
^{2}}A_{i_{k}}^{\ast }((S_{i_{k}}(A_{i_{k}}x^{k})-A_{i_{k}}x^{k})))
\label{e-xk-ac1}
\end{equation}%
or shorter%
\begin{equation}
x^{k+1}=V_{k}x^{k}\text{,}  \label{e-xkac}
\end{equation}%
where $V_{k}=V_{i_{k}}^{\lambda _{k},\mu _{k}}$, $\lambda _{k},\mu _{k}\in
\lbrack \varepsilon ,2-\varepsilon ]$, $k\geq 0$, for some $\varepsilon \in
(0,1)$ and $x^{0}\in \mathcal{H}$ is arbitrary. Recall that a sequence $%
\{i_{k}\}_{k=0}^{\infty }\subseteq I=\{1,2,...,p\}$ is called an almost
cyclic control if there is a constant $s\geq p$ such that for any $k\geq 0$
the set $\{i_{k+1,}i_{k+2},...,i_{k+s}\}\subseteq I$ (cf. \cite[%
Definition 5.6.10]{C2013}). Theorems \ref{t-1} and \ref{t-2} presented below
apply results of \cite{CRZ18}.

\begin{theorem}
\label{t-1}Let $S_{i}$ and $U_{i}$ be weakly regular cutters, $i\in I$, $%
\lambda _{k},\mu _{k}\in \lbrack \varepsilon ,2-\varepsilon ]$, $k\geq 0$,
for some $\varepsilon \in (0,1)$ and $\{i_{k}\}_{k=0}^{\infty }$ be an
almost cyclic control. Then the sequence $x^{k}$ generated by iteration (\ref%
{e-xkac}) converges weakly to some $x^{\ast }\in \Omega $.
\end{theorem}

\begin{proof}
Let $\lambda ,\mu \in (0,2)$. Because $T_{i}$ is a cutter, its $\lambda $%
-relaxation $T_{i,\lambda }$ defined by (\ref{e-Til}) is $\alpha $-SQNE, $%
i\in I$, where $\alpha =(2-\lambda )/\lambda $ (see, e.g., \cite[Theorem
2.1.39]{C2013}) and $\limfunc{Fix}T_{i,\lambda }=A_{i}^{-1}(Q_{i})\neq
\emptyset $, because $A_{i}^{-1}(Q_{i})\neq \emptyset $, $i=1,2,...,p$.
Similarly, $U_{i,\mu }$ is $\beta $-SQNE, where $\beta =(2-\mu )/\mu $. By $%
C_{i}\cap A_{i}^{-1}(Q_{i})\neq \emptyset $ (see (\ref{e-omega})), the
operator $V_{i}^{\lambda ,\mu }$ is $\rho $-SQNE, where $\rho =(\alpha
^{-1}+\beta ^{-1})^{-1}$ and $\limfunc{Fix}V_{i}^{\lambda ,\mu }=C_{i}\cap
A_{i}^{-1}(Q_{i})$ \cite[Corollary 2.1.47 and Theorem 2.1.26(ii)]{C2013}, $%
i\in I$. Thus, one can easily check that $V_{i_{k}}^{\lambda _{k},\mu _{k}}$
is $\rho _{k}$-SQNE, where $\rho _{k}\geq \varepsilon /(4-2\varepsilon )>0$.
Let $n_{-1}^{i}=-1$ and
\begin{equation}
n_{k}^{i}=\min \{n>n_{k-1}^{i}:i_{n}=i\}\text{,}  \label{e-nki}
\end{equation}%
$i\in I$, $k\geq 0$. Because $\{i_{k}\}_{k=0}^{\infty }$ is almost cyclic, $%
n_{k}^{i}$ is well defined and there is $s\geq p$ such that $%
n_{k+1}^{i}-n_{k}^{i}\leq s$ for all $k\geq 0$, $i=1,2,...,p$. The sequences
$\{T_{i,\lambda _{n_{k}^{i}}}\}_{k=0}^{\infty }$ and $\{U_{i,\mu
_{n_{k}^{i}}}\}_{k=0}^{\infty }$ are weakly regular \cite[Proposition 4.7]%
{CRZ18}. By $V_{n_{k}^{i}}=V_{i}^{\lambda _{n_{k}^{i}},\mu
_{n_{k}^{i}}}=U_{i,\mu _{n_{k}^{i}}}T_{i,\lambda _{n_{k}^{i}}}$, the
sequence $\{V_{n_{k}^{i}}\}_{k=0}^{\infty }$ is weakly regular \cite[%
Corollary 5.5]{CRZ18}. Now \cite[Theorem 6.2(i)]{CRZ18} yields the weak
convergence of $x^{k}$ to $x^{\ast }\in \Omega $.
\end{proof}

\bigskip

Setting $S_{i}=P_{Q_{i}}$, $x\in \mathcal{H}$, $U_{i}=P_{C_{i}}$, $\mu _{k}=1
$ and $\{i_{k}\}_{k=0}^{\infty }$ a cyclic control, i.e., $i_{k}=[k]:=k(%
\func{mod}p)+1$, we obtain the following.

\begin{corollary}
Let $\lambda _{k}\in \lbrack \varepsilon ,2-\varepsilon ]$, $k\geq 0$, for
some $\varepsilon \in (0,1)$. Then the sequence $x^{k}$ generated by
iteration
\begin{equation}
x^{k+1}=P_{C_{[k]}}(x^{k}+\frac{\lambda _{k}}{\Vert A_{[k]}\Vert ^{2}}%
A_{[k]}^{\ast }((P_{Q_{[k]}}(A_{[k]}x^{k})-A_{[k]}x^{k})).  \label{e-xk-CQ}
\end{equation}%
converges weakly to some $x^{\ast }\in \Omega $.
\end{corollary}

\begin{proof}
The metric projection, as a nonexpansive operator is weakly regular \cite[%
Lemma 2]{Opi67}. Thus, the Corollary follows directly from Theorem \ref{t-1}.
\end{proof}

\bigskip

Now taking $Q_{[k]}=\{\hat{b}^{[k]}\}$ and $A_{[k]}=\hat{A}_{[k]}$ in (\ref%
{e-xk-CQ}) and assuming $\mathcal{H}_{i}$ being of finite dimension we
retrieve P-BIM and this proves Theorem \ref{main:thm}.

\bigskip

Next suppose that $C_{i}$ is a polyhedral set (the intersection of a finite
family of half spaces), $U_{i}=P_{C_{i}}$, $\mathcal{H}_{i}$ is finite
dimensional, $Q_{i}=\{b^{i}\}$ and $S_{i}=P_{Q_{i}}$, $i\in I$. Then $%
P_{Q_{i}}(y)=b^{i}$, $y\in \mathcal{H}_{i}$, the range of $A_{i}$ is closed,
$A_{i}^{-1}(Q_{i})$ is polyhedral, $i\in I$, and iteration (\ref{e-xk-ac1})
can be written as
\begin{equation}
x^{k+1}=P_{C_{i_{k}},\mu _{k}}(x^{k}+\frac{\lambda _{k}}{\Vert
A_{i_{k}}\Vert ^{2}}A_{i_{k}}^{\ast }(b_{i_{k}}-A_{i_{k}}x^{k}))\text{.}
\label{e-xk-lin}
\end{equation}%
Denote $b=(b^{1},b^{2},...b^{p})\in \mathcal{H}_{0}$. Then $%
A^{-1}(Q)=\{x\in \mathcal{H}:Ax=b\}$ and the solution set has the form $%
\Omega =\{x\in C:Ax=b\}$.

\begin{theorem}
\label{t-2}Let $C_{i}$, $i\in I$, be polyhedral, $\lambda _{k},\mu _{k}\in
\lbrack \varepsilon ,2-\varepsilon ]$, $k\geq 0$, for some $\varepsilon \in
(0,1)$ and $\{i_{k}\}_{k=0}^{\infty }\subseteq I$  be almost cyclic. Then the sequence $%
x^{k}$ generated by iteration (\ref{e-xk-lin}) converges linearly to some $%
x^{\ast }\in \Omega $.
\end{theorem}

\begin{proof}
Similarly as in the proof of Theorem \ref{t-1}, $V_{i_{k}}^{\lambda _{k},\mu
_{k}}$ is $\rho _{k}$-SQNE, where $\rho _{k}\geq \varepsilon
/(4-2\varepsilon )>0$. Moreover, the sequences $\{T_{i,\lambda
_{n_{k}^{i}}}\}_{k=0}^{\infty }$ and $\{U_{i,\mu
_{n_{k}^{i}}}\}_{k=0}^{\infty }$ are linearly regular \cite[Proposition 4.7]%
{CRZ18}. Because, $C_{i}\ $and $A_{i}^{-1}(Q_{i})$ are polyhedral sets, the
family $\{C_{i},A_{i}^{-1}(Q_{i})\}$ is linearly regular, $i\in I$ \cite[%
Corollary 5.26]{BB96}. Thus, the sequences $\{V_{n_{k}^{i}}\}_{k=0}^{\infty }
$, $i\in I$, are linearly regular \cite[Corollary 5.5(iii)]{CRZ18} (for a
definition of a linearly regular sequence of operators, see \cite[Definition
4.1]{CRZ18}). Because $F_{i}:=C_{i}\cap A_{i}^{-1}(Q_{i})$, $i\in I$, are
polyhedral sets, the family $\{F_{i}:i\in I\}$ is linearly regular \cite[%
Corollary 5.26]{BB96}. Now \cite[Theorem 6.2(iii)]{CRZ18} yields the linear
convergence of $x^{k}$ to $x^{\ast }\in \Omega $.
\end{proof}

We end this section by considering the case $N\neq \textrm{Id}$.
For ease of notation  we only consider the simultaneous case.
Let $A:\mathcal{H}_{1}\rightarrow \mathcal{H}_{2}$ be a bounded linear
operator and $b\in \mathcal{H}_{2}$. Define $T_{A,b}:\mathcal{H}%
_{1}\rightarrow \mathcal{H}_{1}$ by
\begin{equation}
T_{A,b}(x)=x+\frac{1}{\Vert A\Vert ^{2}}A^{\ast }(b-Ax)\text{,}
\end{equation}%
$x\in \mathcal{H}_{1}$. The operator $T_{A,b}$
is firmly nonexpansive  (\cite[Definition 2.2.1]{C2013}),
 and \newline $\limfunc{Fix}%
T_{A,b}=\{x\in \mathcal{H}_{1}:A^{\ast }(b-Ax)=0\}$, see, e.g., \cite[Lemma
4.6.2 and Theorem 4.6.3]{C2013}. In particular, if $A^{-1}(\{b\})\neq
\emptyset $ then $\limfunc{Fix}T_{A,b}=A^{-1}(\{b\})$ \cite[Lemma 3.1]{WX11}.%

Let $N:\mathcal{H}_{1}\rightarrow \mathcal{H}_{1}$ and $M:\mathcal{H}%
_{2}\rightarrow \mathcal{H}_{2}$ be strongly positive symmetric operators
and $N^{\frac{1}{2}}$ and $M^{\frac{1}{2}}$ their square roots. It is well
known that $N^{\frac{1}{2}}$ and $M^{\frac{1}{2}}$ are also strongly
positive definite symmetric operators \cite[page 618]{Zei90}. Let $\tilde{A}%
:=M^{\frac{1}{2}}AN^{\frac{1}{2}}$ and $\hat{b}:=M^{\frac{1}{2}}b$. Define $%
\hat{T}:=T_{\tilde{A},\hat{b}\,}$, i.e.,

\[
\hat{T}(x)=x+\frac{1}{\Vert \tilde{A}\Vert ^{2}}\tilde{A}^{\ast }
(\hat{b}-\tilde{A}%
x)=x+\frac{1}{\Vert M^{\frac{1}{2}}AN^{\frac{1}{2}}\Vert ^{2}}N^{\frac{1}{2}%
}A^{\ast }M(b-AN^{\frac{1}{2}}x)\text{.}
\]%
Clearly, $\hat{T}$ is firmly nonexpansive and
\[
\limfunc{Fix}\hat{T}=\{x\in \mathcal{H}_{1}:\hat{A}^{\ast }(\hat{b}-\hat{A}%
x)=0\}=\{x\in \mathcal{H}_{1}:N^{\frac{1}{2}}A^{\ast }M(b-AN^{\frac{1}{2}%
}x)=0\}\text{.}
\]

Let $\hat{T}_{\lambda }$ denote the $\lambda $-relaxation of $\hat{T}$,
where $\lambda \in \lbrack 0,2]$, and recall the definition of an averaged
operator (\cite[Definition 2.2.16]{C2013}).
\begin{proposition}
The operator $V:\mathcal{H}_{1}\rightarrow \mathcal{H}_{1}$ defined by $%
V=P_{C}\hat{T}_{\lambda }$, where $\lambda \in (0,2)$, is $\frac{2}{%
4-\lambda }$-averaged.
\end{proposition}

\begin{proof}
$V$ is $\frac{4}{4-\lambda }$-relaxed firmly nonexpansive being a
  composition
of a firmly nonexpansive operator $P_{C}$ and a $\lambda $-relaxed firmly nonexpansive
operator $\hat{T}_{\lambda }$ \cite[Theorem 2.2.37]{C2013}. Consequently, $V$
is $\frac{2}{4-\lambda }$-averaged \cite[Corollary 2.2.17]{C2013}.
\end{proof}

Consider now the iterates generated by $x^{k+1}=Vx^{k}.$
It follows that
$x^{k}\rightharpoonup x^{*}\in \limfunc{Fix}V$ provided $\limfunc{Fix}V
\neq \emptyset.$

\section{P-BIM and its noise error}\label{SIRT:Sec}
 In this section we will study
how errors in the unperturbed
right-hand-side
$\bar{b}$ affect the iterates in P-BIM. To this end let
\begin{equation}
  b=\bar{b}+\delta b.
\end{equation}
The noise will apart from the iterates also  possibly affect the relaxation
parameters (depending on their definition).
To cope with possible rank-deficiency we will, similarly as in
\cite{EHN2012,NK2015}
add a regularization term $\alpha\|x\|^{2}$.
%
Let
\begin{equation} \label{uk}
u^k(x,b^{[k]})= A_{[k]}^T M_{[k]}\left(b^{[k]}-A_{[k]} x\right)-\alpha x,
\end{equation}
where the last term is due to regularization.
Next we introduce the notations
\begin{equation}
\theta_k=\frac{\lambda_k}{\|\hat{A}_{[k]}\|^{2}+\alpha}, \
\bar{\theta}_k=\frac{\bar{\lambda}_k}{\|\hat{A}_{[k]}\|^{2}+\alpha}.
\end{equation}
We will in the sequel refer to
$\theta_k$ and $\bar{\theta}_k$ as the noisy and noise-free relaxation parameter,
respectively. Note that these are just a scaled version of the original
relaxation parameters $\lambda_k$ and $\bar{\lambda}_k$. For noise-free data
$\lambda_k=\bar{\lambda}_k$. Now the noise-free P-BIM becomes
\begin{equation} \label{P-BIM1}
\bar{x}^{k+1}=P_{C_{[k]}}\left(\bar{x}^{k}+
\bar{\theta}_k u^{k}(\bar{x}^{k},\bar{b}^{[k]})
\right).
\end{equation}
The noisy version of P-BIM is
\begin{equation} \label{P-BIMN}
x^{k+1}=P_{C_{[k]}}\left(x^{k}+\theta_k u^{k}(x^{k},b^{[k]})\right).
\end{equation}

Next define
\begin{equation}\label{def1}
\hat{\delta}=\max_{1\leq t\leq p}\|A_{t}^T M_{t}\delta b^{t}\|,
\end{equation}
\begin{equation}
\gamma_{k}=\max_{0\leq j\leq k} |\bar{\theta}_{j}-\theta_{j}|, \
\underline{\eta}_k= \min_{0\leq j\leq k} \theta_{j}, \
 \overline{\eta}_k= \max_{0\leq j\leq k} \theta_{j}.
\end{equation}

Let $\underline{\sigma}_t$ be the smallest \textit{nonzero} singular value
of $M_t^{1/2}A_{t}$, and define
\begin{equation} \label{singval}
\underline{\sigma}=\min_{1\leq t\leq p}\underline{\sigma}_t, \quad
\overline{\sigma}=\max_{1\leq t\leq p}\|M_t^{1/2}A_t\|,
\end{equation}
and
\begin{equation}\label{hat_u}
\hat{u}_k=\max_{0\leq s\leq k}\|u^{s}(\bar{x}^{s},\bar{b}^{[s]})\|.
\end{equation}

\begin{theorem}\label{upper:bound:block}
Assume that $\alpha=\underline{\sigma}^{2}$.
Then the  noise-error in P-BIM 
is bounded above by
\begin{equation}\label{UpperBl}
e^{k}_{N}:=\|x^{k}-\bar{x}^{k}\|\leq\frac{1}{\underline{\sigma}}
\left(\frac{\gamma_{k-1}\hat{u}_{k-1}	+
\overline{\eta}_{k-1}\hat{\delta}}{\underline{\eta}_{k-1}}\right)\Psi^{k}
(\underline{\sigma},\underline{\eta}_{k-1}),
\end{equation}
where $\Psi^{k}(x,y)$ is 
 defined by
\begin{equation}\label{P:si}
\Psi^{k}(x,y)\equiv \frac{1-(1-yx^{2})^{k}}{x}.
\end{equation}
\end{theorem}
\begin{proof}
 Since $P_{C}$ is  nonexpansive we get
\begin{align}
e^{k}_{N}
&\leq  \|x^{k-1}+\theta_{k-1}u^{k-1}(x^{k-1},b^{[k]})
-
\left(\bar{x}^{k-1}+
\bar{\theta}_{k-1}u^{k-1}(\bar{x}^{k-1},\bar{b}^{[k]})\right)\|
\nonumber
\\
&=
\|(x^{k-1}-\bar{x}^{k-1})+\theta_{k-1}(u^{k-1}(x^{k-1},b^{[k]})-
u^{k-1}(\bar{x}^{k-1},\bar{b}^{[k]})) \nonumber
\\
&+
(\theta_{k-1}-\bar{\theta}_{k-1})u^{k-1}(\bar{x}^{k-1}, \bar{b}^{[k]})\|.
\nonumber
\end{align}
Now
\begin{align}
u^{k-1}(x^{k-1},b^{[k]})-u^{k-1}(\bar{x}^{k-1},\bar{b}^{[k]})&= A_{[k-1]}^T M_{[k-1]}\left(b^{[k-1]}-A_{[k-1]} x^{k-1}\right)\nonumber \\
&\quad -A_{[k-1]}^T M_{[k-1]}\left(\bar{b}^{[k-1]}-A_{[k-1]} \bar{x}^{k-1}\right)\nonumber \\
&\quad  +\alpha(\bar{x}^{k-1}-x^{k-1}) \nonumber \\
&= A_{[k-1]}^T M_{[k-1]}\delta b^{[k-1]}\nonumber\\
&\quad +(\alpha \textrm{Id}+ A_{[k-1]}^T M_{[k-1]}A_{[k-1]})
(\bar{x}^{k-1}-x^{k-1}). \nonumber
\end{align}
Hence
\begin{eqnarray}
e^{k}_{N} &\leq&
\|\left((1-\theta_{k-1}\alpha)\textrm{Id}-\theta_{k-1}A_{[k-1]}^T M_{[k-1]}A_{[k-1]}\right)
(x^{k-1}-\bar{x}^{k-1}) \nonumber \\
&+&
\theta_{k-1}A_{[k-1]}^T M_{[k-1]}\delta b^{[k-1]}+
(\theta_{k-1}-\bar{\theta}_{k-1})
u^{k-1}(\bar{x}^{k-1}, \bar{b}^{[k-1]})\|.
\nonumber
\end{eqnarray}
Here
\begin{equation}
Q_{k}=\left((1-\theta_{k}\alpha)\textrm{Id}-
\theta_{k} A_{[k]}^{T}M_{[k]}A_{[k]}\right), \quad
q_{k}=\|Q_{k}\|.
\end{equation}

Then it follows
\[
e^{k}_{N} \leq
 q_{k-1} e_N^{k-1} + |\theta_{k-1}-\bar{\theta}_{k-1}|.
\|u^{k-1}(\bar{x}^{k-1},\bar{b}^{[k-1]})\|+\theta_{k-1} \hat{\delta}.
\]

 Assuming $e^{0}_N=x^{0}-\bar{x}^{0}=0$
it follows by induction
\begin{align}
e^{k}_{N}&\leq \sum^{k-2}_{s=0}|\theta_{s}-\bar{\theta}_{s}|.\|u^{s}(\bar{x}^{s},\bar{b}^{[s]})
\|\prod_{j=s+1}^{k-1}q_{j}\nonumber\\
&+ \hat{\delta} \sum^{k-2}_{s=0}\theta_{s}\prod_{j=s+1}^{k-1}q_{j} +
|\theta_{k-1}-\bar{\theta}_{k-1}|.\|u^{k-1}(\bar{x}^{k-1},\bar{b}^{[k-1]})\| +
\theta_{k-1} \hat{\delta}.
\nonumber \end{align}
With (\ref{hat_u}), and putting $\hat{q}_k=\max_{0\leq j\leq k}q_j$ we therefore get
\begin{align}
e^{k}_{N}&\leq (\gamma_{k-2}\hat{u}_{k-2}+\bar{\eta}_{k-2}\hat{\delta})
\sum^{k-2}_{s=0}\prod_{j=s+1}^{k-1}\hat{q}_{k-1}
+(\gamma_{k-1}\hat{u}_{k-1}+\bar{\eta}_{k-1}\hat{\delta}).
\end{align}
Now
\begin{equation} \label{trivb}
\sum^{k-2}_{s=0}
\prod_{j=s+1}^{k-1}\hat{q}_{k-1} =
\left(\hat{q}_{k-1}^{k-1}+\hat{q}_{k-1}^{k-2}+\ldots+\hat{q}_{k-1}^{2}+
\hat{q}_{k-1}^{1}\right).
\end{equation}
It follows from the properties of geometric progression that,
\begin{equation*}
e^{k}_{N}\leq \left(\gamma_{k-1}\hat{u}_{k-1}+\overline{\eta}_{k-1}\delta\right)
\frac{1-\hat{q}_{k-1}^{k}}{1-\hat{q}_{k-1}}.
\end{equation*}
We will now use the following result, derived
 in
\cite[Lemma 1]{NK2015}, for block-iteration, and \cite[Lemma 3.9]{EHN2012}
for simultaneous iteration,
\begin{equation}\label{q1-form}
\|Q_k\|=1-\theta_k\underline{\sigma}^2,
\textrm{ assuming } \alpha=\underline{\sigma}^2.
\end{equation}
Using (\ref{q1-form}) it holds
$\hat{q}_{k-1}=\max_{0\leq j\leq k-1}(1-\theta_j\underline{\sigma}^{2})
=1-\underline{\eta}_{k-1}\underline{\sigma}^{2}.$
It follows
$$
e^{k}_{N}\leq 
\left(\gamma_{k-1}\hat{u}_{k-1}+
\overline{\eta}_{k-1}\hat{\delta}\right )
\frac{1-(1-\underline{\eta}_{k-1}\underline{\sigma}^{2})
^{k}}{\underline{\eta}_{k-1}\underline{\sigma}^{2}}.
$$
Therefore the result follows using (\ref{P:si}).
\end{proof}

\begin{remark}
In the proof of Theorem \ref{upper:bound:block}, we assumed cyclic control. However, provided the same control sequence $\{i_k\}$ (where $1\leq i_k \leq p$) is used in both the noise-free (\ref{P-BIM1}), and the noisy (\ref{P-BIMN}) iteration it is obvious that Theorem \ref{upper:bound:block} also holds.  Examples of other controls are almost cyclic control as defined above and stochastic control \cite[Definition 5.1]{stoch2021}. Stochastic  control has proven to be quite successful and we refer to \cite{RPT2020,Ne2019,NP2014} for some recent developments.
\end{remark}

\begin{remark}
When the regularization parameter $\alpha$ is chosen positive
the convergence of P-BIM follows from the contraction mapping theorem
as follows.
By (\ref{q1-form}) $Q_k$ is a $\rho_k$-contraction with
$\rho_k=1-\theta_k\underline{\sigma}^{2}.$
Since $P_C$ is a nonexpansive it follows that $U_k$ (defined in Algorithm \ref{BIP:algt})
also is a contraction.
Because $\rho_k<\rho<1,$ it follows that $U_k$ is a
$\rho-$contraction. Thus, assuming feasibility, $x^{k}$ converges to
$x^{*}\in \cap_{k=1}^{\infty} \limfunc{Fix}U_k$.
\end{remark}

\begin{remark}
If $\lambda_k=\bar{\lambda}_k$, i.e. the relaxation parameters do not depend
on the noise, then $\gamma_k=0$, and
$\underline{\eta}_k=\min_{0\leq j \leq k} \theta_{j}, \overline{\eta}_k=\max_{0\leq j \leq k} \theta_{j}$
so that
\begin{equation}\label{UpperNew}
  \|x^k-\bar{x}^k\|\leq \frac{1}{\underline{\sigma}}\left(\frac{\overline{\eta}_{k-1}\hat{\delta}}{\underline{\eta}_{k-1}}\right)\Psi^{k}(\underline{\sigma},\underline{\eta}_{k-1})
\end{equation}
This bound, assuming decreasing relaxation parameters $\theta_{k}$,
coincides with the bound given in \cite[Theorem 3]{NK2015}.

Next we remind the reader of the inequality
$(1-(1-x^2)^k)/x \leq \sqrt{k}, x \in (0,1)$.
It follows
\begin{equation*}
  \Psi^{k}(\underline{\sigma},\underline{\eta}_{k-1})=\frac{1-(1-\underline{\eta}_{k-1}\underline{\sigma}^{2})^{k}}{\sqrt{\underline{\eta}_{k-1}}\underline{\sigma}}
\sqrt{\underline{\eta}_{k-1}}\leq \sqrt{\underline{\eta}_{k-1}}\sqrt{k},
\end{equation*}
provided $\underline{\sigma}\leq 1/\underline{\eta}_{k-1}$. Hence we arrive at the bound (with $c$ a small constant)

\begin{equation}
\|x^{k}-\bar{x}^{k}\|\leq \frac{c}{\underline{\sigma}}\sqrt{k},~~~~~~\underline{\sigma}\leq 1/\underline{\eta}_{k-1},
\end{equation}
cf \cite{Elfving2014,KL-14} for Kaczmarz’s method and \cite{engl1996regularization} for SIRT. Due to the factor $1/\underline{\sigma}$
this bound is unrealistically large, but we have not been able to sharpen it. See also \cite{elfving2018row} for a more detailed discussion.
\end{remark}

We next give an alternative upper bound for the noise error presented in
Theorem \ref{upper:bound:block}, when the relaxation parameters are
decreasing, i.e. $0< \theta_{k+1}\leq \theta_{k}$.
At first, following \cite[Propositions 2.3, 2.4]{elfving2010}, we consider
\begin{equation}\label{zeta:equ}
g_{k-1}(y)=(2k-1)y^{k-1}-(y^{k-2}+\cdots+y+1)
\end{equation}
which has a unique root $\zeta_{k} \in (0,1)$ and satisfies $0 < \zeta_{k} < \zeta_{k+1} < 1$
and $\lim_{k\to \infty}\zeta_{k} = 1$.
\begin{theorem}\label{thm:dec:relax}
If the relaxation parameters are decreasing and satisfy
\begin{equation}\label{lambda:less:sigma}
 0< \theta_{k}\leq \frac{1}{\|M_t^{1/2}A_{t}\|^{2}},
\end{equation}
then we have
\begin{equation}\label{decreasing:relaxation}
\|x^{k}-\bar{x}^{k}\|\leq\displaystyle\frac{1}
{\underline{\sigma}}\left(\frac{(\theta_{0}-\theta_{k-1})
\hat{u}_{k-1}+\theta_{0}\hat{\delta}}
{{\sqrt{\theta_{k-1}}}}\right)\frac{1-{\zeta}^{k}_{k}}{\sqrt{1-\zeta_{k}}}.
\end{equation}
\end{theorem}
\begin{proof}
It holds
$\underline{\eta}_{k-1}=\min \theta_s=\theta_{k-1}$,
$\overline{\eta}_{k-1}=\max \theta_s=\theta_{0}$ and
$\gamma_{k-1}\leq|\theta_{0} - \theta_{k-1}|$. Based on
\cite[Proposition 2.4]{elfving2010}, we have
\begin{equation}\label{max:Psi}
\displaystyle\max_{0<\sigma\leq 1/{\sqrt{\theta_{k-1}}}}
\Psi^{k}(\sigma,\theta_{k})\leq \sqrt{\theta_{k-1}}
\frac{1-\zeta^{k}_{k}}{\sqrt{1-\zeta_{k}}}
\end{equation}
where $\zeta_{k} \in (0,1)$ is the unique root of (\ref{zeta:equ}).
By applying (\ref{max:Psi}) in (\ref{UpperBl}) we obtain
(\ref{decreasing:relaxation}).
\end{proof}

\section{Relaxation parameters}\label{relaxation:par}

%
In \cite{NK2015} 
three step-size rules for
use in block iterative methods were studied
(the first two  were originally proposed in
\cite{elfving2010}, and in \cite{EHN2012} (constrained case)
for simultaneous iteration). The main idea with these rules is to control
the propagated noise component of the error. Here we use the formulation
(\ref{P-BIM1}) for P-BIM (to conform with the notations in the preceding
Section and in \cite{elfving2010}, \cite{EHN2012}).
%
 %
\begin{equation}\label{Psi1:Block}
\Psi_{1}-\textrm{rule}~~~~~~~~~~~ \theta_{k} = \left\{ \begin{array}{ll}
\displaystyle \sqrt{2}\overline{\sigma}^{-2},& ~~~~~~~~~~ \textrm{for}~~ k=0,1\vspace{0.1cm} \\
\displaystyle 2\overline{\sigma}^{-2}(1-\zeta_{k}), & ~~~~~~~~~~\textrm{for}~~ k\geq 2,
\end{array} \right.
\end{equation}
\begin{equation}\label{Psi2:Block}
\Psi_{2}-\textrm{rule}~~~~~~~~~~~ \theta_{k} = \left\{ \begin{array}{ll}
\displaystyle \sqrt{2}\overline{\sigma}^{-2},& ~~~~~~~~~~ \textrm{for}~~ k=0,1\vspace{0.1cm} \\
\displaystyle 2\overline{\sigma}^{-2}\frac{1-\zeta_{k}}{(1-\zeta_{k}^{k})^{2}},
& ~~~~~~~~~~\textrm{for} ~~k\geq 2,
\end{array} \right.
\end{equation}
\begin{equation}\label{Psi3}
\Psi_{3}-\textrm{rule}~~~~~~~~~~~ \theta_{k} = \left\{ \begin{array}{ll}
\displaystyle \sqrt{2}\overline{\sigma}^{-2},& ~~~~~~~~~~ \textrm{for}~~ k=0,1\vspace{0.1cm} \\
\displaystyle 2\overline{\sigma}^{-2}\frac{(1-\zeta_{k}^{k})^{2}}
{(1-\zeta_{k})^{1-r}}, & ~~~~~~~~~~\textrm{for} ~~k\geq 2.
\end{array} \right.
\end{equation}
Here $1<r\leq 2.$
All three rules are descending, i.e.
$0 < \theta_{k+1} < \theta_k, \ k \geq 2$. Also for these three rules the relaxation parameters
do not depend on the noise.
%


The following upper bounds for the noise error were derived in \cite{NK2015}
(for the simultaneous case see also \cite{EHN2012}).

\begin{equation}\label{Psi_upper}
\left\|x^{k}-\bar{x}^{k}   \right\|\leq  \left\{ \begin{array}{ll}
\displaystyle \frac{\beta_{\delta b}}{\underline{\sigma}}
\left(1-\zeta_{k}^{k}\right)/\left(1-\zeta_{k}\right),& ~~~~~~~ \Psi_{1}-\textrm{rule}\vspace{0.1cm} \\
\displaystyle \frac{\beta_{\delta b}}{\underline{\sigma}}\left(1-\zeta_{k}^{k}\right)^{2}/\left(1-\zeta_{k}\right),& ~~~~~~~ \Psi_{2}-\textrm{rule}\vspace{0.1cm} \\
\displaystyle \frac{\beta_{\delta b}}{\underline{\sigma}}\left(1-\zeta_{k}\right)^{\frac{-r}{2}}.& ~~~~~~~ \Psi_{3}-\textrm{rule}\vspace{0.1cm}
\end{array} \right.
\end{equation}
Here
\begin{equation}
\beta_{\delta b}=
\max_{1\leq t\leq p} \|M_{[t]}^{\frac{1}{2}}\delta b^{[t]}\|, \
  \beta_{b}=
\max_{1\leq t\leq p} \|M_{[t]}^{\frac{1}{2}}b^{[t]}\|.
\end{equation}
We next derive a new relaxation parameter rule which will be based
on the bound (\ref{decreasing:relaxation}).
We will then use the following heuristic
 $ \hat{u}_k\leq \hat{u}_0.$ This was fulfilled in all our numerical
experiments (but we have no formal proof). Also we use, for simplicity
$x^{0}=0$ (as in the experiments).
Hence
$$
\hat{u}_k\approx \hat{u}_0\leq \max_{1\leq t \leq p}\|A_tM_tb^{t}\|\leq
\overline{\sigma}\beta_{b}.
$$
Further using (\ref{def1})
$
\hat{\delta}
\leq \overline{\sigma}\beta_{\delta b}.
$
Summarizing we get
\begin{equation} \label{heuristic}
\hat{u}_k\approx  \bar{\sigma}\beta_{b}, \
\hat{\delta}\approx\bar{\sigma}\beta_{\delta b}.
\end{equation}
Again the basic idea is to control the noise error.
Therefore consider the function
\begin{equation}\label{gam}
\frac{\beta_{\delta b}}{\underline{\sigma}}
\left(1-\zeta_{k}^{k}\right)^{\frac{1}{2}}\left(1-\zeta_{k}
\right)^{\frac{-r}{2}}
\end{equation}

We  now enforce that the noise error bound (\ref{decreasing:relaxation})
equals
(\ref{gam}). It follows, also using
(\ref{heuristic})
$$ 
\frac{\overline{\sigma}}{\underline{\sigma}}\left(\frac{(\theta_{0}-
\theta_{k-1})\beta_{b}+\theta_{0}\beta_{\delta b}}
{{\sqrt{\theta_{k-1}}}}\right)\frac{1-{\zeta}^{k}_{k}}{\sqrt{1-\zeta_{k}}}=
\frac{\beta_{\delta b}}{\underline{\sigma}}
\left(1-\zeta_{k}^{k}\right)^{\frac{1}{2}}\left(1-\zeta_{k}
\right)^{\frac{-r}{2}}
$$
After simplification we get
%
$$
\overline{\sigma}
\left(\frac{(\theta_{0}-
\theta_{k-1})\beta_{b}+\theta_{0}\beta_{\delta b}}
{\sqrt{\theta_{k-1}}}\right)=
\beta_{\delta b}\frac{(1-\zeta_{k})^{\frac{1-r}{2}}}
{\sqrt{1-{\zeta}^{k}_{k}} }.
$$
By solving for $\theta_{k-1}$, and using that
 $\theta_0=\sqrt{2}/\overline{\sigma}^{2}$
 one gets (after elementary calculations)
$$\theta_{k-1}=
\frac{\mathcal{B}+\mathcal{Z}_{k,r}^{2}\beta_{\delta b}^{2}   -\mathcal{Z}_{k,r}\beta_{\delta b}
\sqrt{\mathcal{Z}_{k,r}^{2}\beta_{\delta b}^{2}+2\mathcal{B}}}{2\overline{\sigma}^2\beta_{b}^2},
$$
where
$$
\mathcal{B}=2\sqrt{2}\beta_b\left(\beta_{b}+\beta_{\delta b}\right), \
\mathcal{Z}_{k,r}=\frac{(1-\zeta_{k})^{\frac{1-r}{2}}}
{\sqrt{(1-\zeta_{k}^{k})}}.
$$
Summarizing we have the following rule
\begin{equation}\label{Gamma2bar}
\Gamma-\textrm{rule}~~~~~~~~ \theta_{k} = \left\{ \begin{array}{ll}
\displaystyle \frac{\sqrt{2}}{\overline{\sigma}^{2}},& ~~~~~~~ \textrm{for}~~ k=0,1\vspace{0.1cm} \\
\displaystyle \frac{\mathcal{B}+\mathcal{Z}_{k,r}^{2}\beta_{\delta b}^{2}   -\mathcal{Z}_{k,r}\beta_{\delta b} \sqrt{\mathcal{Z}_{k,r}^{2}\beta_{\delta b}^{2}+2\mathcal{B}}}{2\overline{\sigma}^2\beta_{b}^2}, & ~~~~~~~\textrm{for}~~ k\geq 2,
\end{array} \right.
\end{equation}
where $1< r \leq 2$.
Note that the parameters depend on the noise. With no noise present
$\theta_k=\theta_0, \ k\geq 1$.
We mention here the recent paper \cite[Proposition 5]{Bai-Buc}
where also knowledge of $\|\delta b\|$ is utilized.
By construction the noise upper bound for the $\Gamma-$rule
is
\begin{equation}\label{Gamma_upper}
\|x^{k}-\bar{x}^{k}\|\leq
\frac{\beta_{\delta b}}{\underline{\sigma}}
\left(1-\zeta_{k}^{k}\right)^{\frac{1}{2}}
\left(1-\zeta_{k}\right)^{\frac{-r}{2}}, ~~~~~~~ \Gamma-\textrm{rule}.
\end{equation}
Let $||x^{k}-\bar{x}^{k} ||_{\Omega}$ denote the {\it upper}
bound for
the noise error when $\Omega$-rule is used for choosing relaxation parameters.
Then the following inequalities obviusly hold,
\begin{equation}\label{upper:property}
  ||x^{k}-\bar{x}^{k} ||_{\Gamma}\leq ||x^{k}-\bar{x}^{k}||_{\Psi_{3}}
\textrm{ and }
 ||x^{k}-\bar{x}^{k}||_{\Psi_{2}}\leq ||x^{k}-\bar{x}^{k}||_{\Psi_{1}}~~~~~
\textrm{for}~~~ k\geq 0.
\end{equation}
The bounds for the $\Psi_3$ and the $\Psi_2$ rules respectively may intersect
see \cite[Figure 2]{NK2015}.
We next show that the $\Gamma-$rule also is a diminishing step-size strategy,
\begin{lemma} Let $r\in (1,2], \ k\geq 2$.
 The relaxation parameters in the $\Gamma$-rule
are descending, i.e.
$0 < \theta_{k+1} < \theta_k$.
\end{lemma}
\begin{proof}
  Since $\zeta_{k}\in(0,1)$, we have
  \begin{equation*}
    \frac{d\mathcal{Z}_{k,r}}{d \zeta_{k}}=\left(\frac{r-1}{(1-\zeta_{k})}+
\frac{k \zeta_{k}^{k-1}}{\left(1-\zeta_{k}^{k}\right)}\right)
\frac{(1-\zeta_{k})^{\frac{1-r}{2}}}{2\sqrt{1-\zeta_{k}^{k}}}>0.
  \end{equation*}
  Showing that $\mathcal{Z}_{k,r}(\zeta_k)$ is increasing.
Next
\begin{equation*}
    \frac{d\mathcal{\theta}_{k}}{d \mathcal{Z}_{k,r}}=
\frac{1}{2\overline{\sigma}^2\beta_{b}
\sqrt{\beta_{\delta b}^{2}\mathcal{Z}_{k,r}^{2}+2\mathcal{B}}}
\tau_{k,r},
  \end{equation*}
  with
  \begin{equation*}
    \tau_{k,r}=\beta_{\delta b}\mathcal{Z}_{k,r}
\sqrt{\beta_{\delta b}^{2}\mathcal{Z}_{k,r}^{2}+2\mathcal{B}}-
\left(\mathcal{B}+\beta_{\delta b}^{2}\mathcal{Z}_{k,r}^{2}\right).
  \end{equation*}
We now show that $\tau_{k,r} < 0$. To this end assume the contrary, i.e.
$$
\beta_{\delta b}\mathcal{Z}_{k,r}
\sqrt{\beta_{\delta b}^{2}\mathcal{Z}_{k,r}^{2}+2\mathcal{B}}\geq
\left(\mathcal{B}+\beta_{\delta b}^{2}\mathcal{Z}_{k,r}^{2}\right),
$$
After squaring and simplifying one gets
${\mathcal{B}}^{2}\leq 0,$ a contradiction. Hence
$\frac{d\mathcal{\theta}_{k}}{d \mathcal{Z}_{k,r}}< 0.$
As noted previously $\zeta_k(k)$ is increasing. Therefore the result
follows by the chain rule.
\end{proof}

In Section \ref{Numerical}, in addition to the above rules
also the '$\theta-$opt'  rule is used.
This means finding that constant value of $\theta$
which give rise to the smallest relative error within a fixed number
of cycles ($cmax$). Here a cycle denotes one pass through all
$p$ row-blocks.
This value of $\theta$ is found by searching over the interval
$0<\theta_{opt}<2/\overline{\sigma}^2$.
This strategy requires knowledge of
the exact solution, so for real data one would first need to train
the algorithm using simulated data, see \cite[Section 6]{HANSEN2018air}.
In our tests we took $cmax=100$ and used the exact phantom for training.
This is clearly unrealistic but we think it is of interest to include this rule for comparison with our other rules.
\section{Numerical Results}\label{Numerical}

We will report on tests using  examples from the field of image
reconstruction from projections.
To create the projection matrix $A$ and the right hand side  $b$
the {\text{paralleltomo}} function in the MATLAB package AIR tools
\cite{HANSEN2018air} is used.
We take the Shepp-Logan phantom as
 the original image $x^{*}$ discretized
into  $365 \times 365$ square pixels.
In the first test problem we use $88$ views uniformly distributed
over $180$ degrees, and $516$ projections per view.
Since zero rows do not contribute to the reconstruction,
after identifying and removing these, the resulting matrix $A$
has  dimension $40796 \times 133225$ ({\it case-one}).
Although in our application
iterative methods  usually are more competitive the more  underdetermined
the system is 
this might not be the
case in other applications. Therefore we also consider taking more
rays (264) per projection leading to a matrix of dimension
$122388\times 133225$ ({\it case-two}).
 Apart from using
noise-free data we added  independent Gaussian noise of mean $0$ and
 relative noise-level ($\|\delta b\|/\|b\|$) $2\%$ and $5\%$
respectively.
 In the experiments we set $x^{0}=0$ and $C_{[t]}=[0,1]^{n}$
for $t=1, \ldots, p$.
We further partition the matrix $A$ and the
 right hand side $b$ into $8$ and $22$ blocks, respectively.
The largest and smallest (nonzero) singular values
of each block is estimated using the {\it{power method}} \cite{saad2003}.
We further used
    Cimmino's $M$-Matrix.
The error metric is
{\it{Relative Error}}  defined by
\begin{equation}
  \mbox{Relative Error}=\frac{\|x^{k}-x^{*}\|}{\|x^{*}\|}.
\end{equation}
All codes are written in MATLAB(R2015a) and conducted on a PC with a
Intel Core i7-7700K CPU @ 4.2 GHz and 16 GB RAM.

When using the $\Gamma-$rule we need to estimate $\|M^{\frac{1}{2}}\delta b\|.$
Therefore we randomly generated a new vector
$\overline{\delta b}$ analogous with original $\delta b$.
To show that the estimate $\|M^{\frac{1}{2}}\overline{\delta b}\|$
does not strongly influence the final results, for each case,
three random $\overline{\delta b}$ with different noise levels are generated
as follows.
We first generated
$ e= randn(size(b))$, then put $\overline{\delta b}= g\|b\|e/\|e\|$
where $g$ denotes guessed noise level. We take $g=0.01, 0.02$ and $0.03$
for $2\%$ noise and $g=0.03, 0.05$ and $0.07$  for  $5\%$ noise.
Table \ref{tableX1} 
shows exact
(i.e. $\|M^{\frac{1}{2}}\delta b\|$)
and estimated values (i.e. $\|M^{\frac{1}{2}}\overline{\delta b}\|$).
It will be seen  from Figures below 
that the error-curves are quite robust versus the estimates of the noise.

\begin{table}[]
\caption{Exact and estimated values for $\|M^{\frac{1}{2}}\delta b\|$.}
\begin{tabular}{ccc|c|c|c|}
\cline{4-6}
\multicolumn{1}{l}{}                                 & \multicolumn{1}{l}{}              & \multicolumn{1}{l|}{}                 & \multicolumn{3}{c|}{\textbf{Estimated value}}                                                                   \\ \hline
\multicolumn{1}{|c|}{\textbf{Problem/ noise}}                 & \multicolumn{1}{c|}{\textbf{\# of blocks}} & \textbf{Exact value} & \textbf{g=1\%}                      & \textbf{g=2\%}                      & \textbf{g=3\%}                      \\ \hline
\multicolumn{1}{|c|}{\multirow{2}{*}{case-one/ 2\%}} & \multicolumn{1}{c|}{8}            & 14.04                                 & 5.07                                & 10.14                               & 15.22                               \\ \cline{2-6}
\multicolumn{1}{|c|}{}                               & \multicolumn{1}{c|}{22}           & 19.71                                 & 8.70                                & 17.40                               & 26.10                               \\ \hline
\multicolumn{1}{|c|}{\multirow{2}{*}{case-two/ 2\%}} & \multicolumn{1}{c|}{8}            & 8.75                                  & 4.87                                & 9.56                                & 14.33                               \\ \cline{2-6}
\multicolumn{1}{|c|}{}                               & \multicolumn{1}{c|}{22}           & 15.82                                 & 7.59                                & 15.17                               & 22.77                               \\ \hline
\multicolumn{1}{l}{}                                 & \multicolumn{1}{l}{}              & \multicolumn{1}{l|}{}                 & \multicolumn{1}{l|}{\textbf{g=3\%}} & \multicolumn{1}{l|}{\textbf{g=5\%}} & \multicolumn{1}{l|}{\textbf{g=7\%}} \\ \hline
\multicolumn{1}{|c|}{case-one/ 5\%}                  & \multicolumn{1}{c|}{8}            & 28.36                                 & 19.29                               & 32.14                               & 45.01                               \\ \cline{2-6}
\multicolumn{1}{|c|}{}                               & \multicolumn{1}{c|}{22}           & 45.91                                 & 31.55                               & 52.59                               & 73.64                               \\ \hline
\multicolumn{1}{|c|}{\multirow{2}{*}{case-two/ 5\%}} & \multicolumn{1}{c|}{8}            & 15.10                                 & 8.95                                & 14.91                               & 20.88                               \\ \cline{2-6}
\multicolumn{1}{|c|}{}                               & \multicolumn{1}{c|}{22}           & 21.57                                 & 13.48                               & 22.47                               & 31.46                               \\ \hline
\end{tabular}
\label{tableX1}
\end{table}

Following \cite{NK2015}, we take $r=1.5$ in the $\Psi_{3}$ rule.
As noted above $r\in(1,2]$ in the $\Gamma-$rule.
\begin{figure}
    \centering \label{r-fig-iter}
\subfigure[]{\includegraphics[width=0.45\textwidth, height=4cm,
trim={1cm .75cm  1cm  0.5cm },clip]
{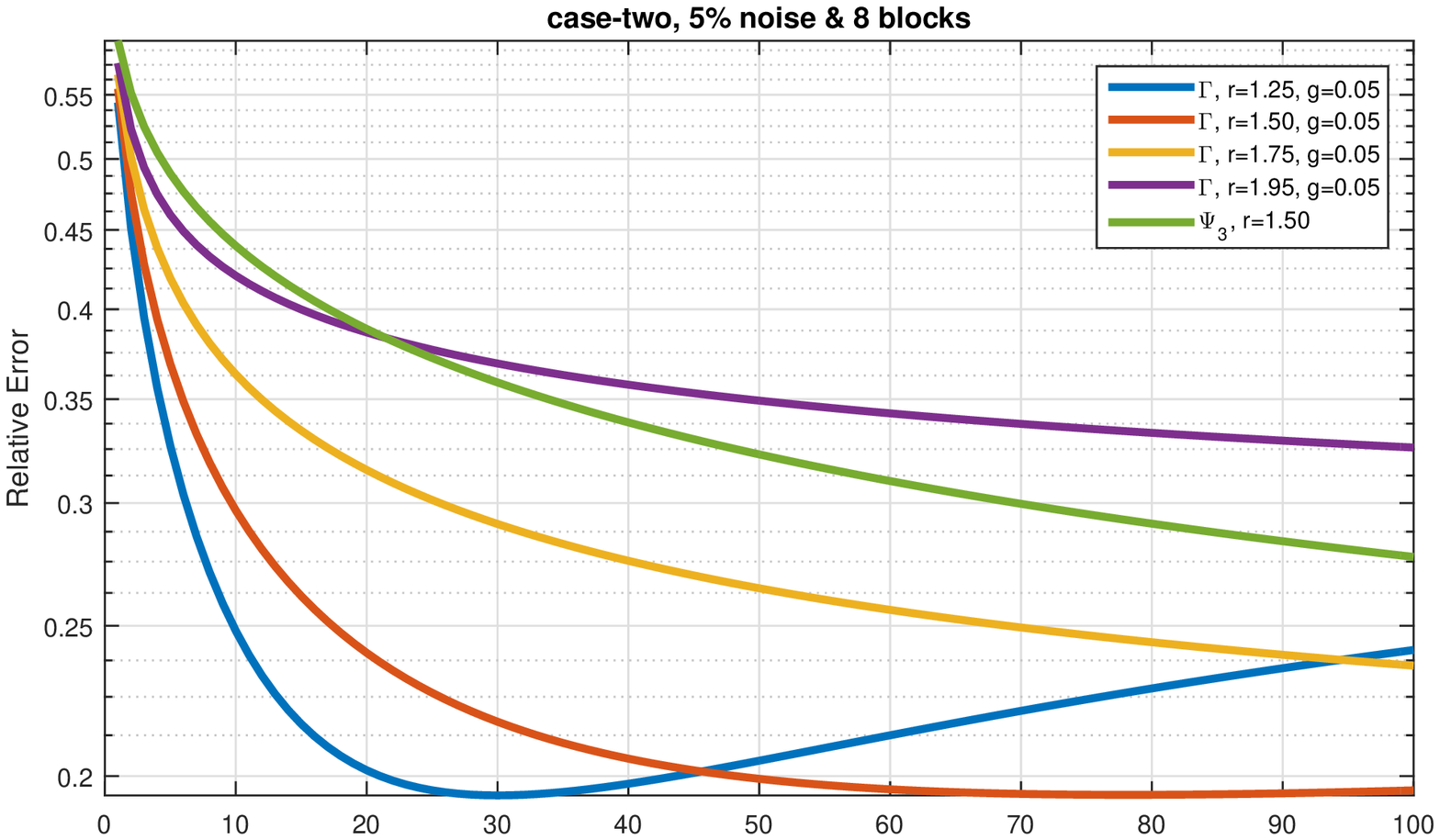}}
    \quad
\subfigure[]{\includegraphics[width=0.45
\textwidth, height=4cm, trim={1cm .75cm  1cm  0.5cm },clip]
{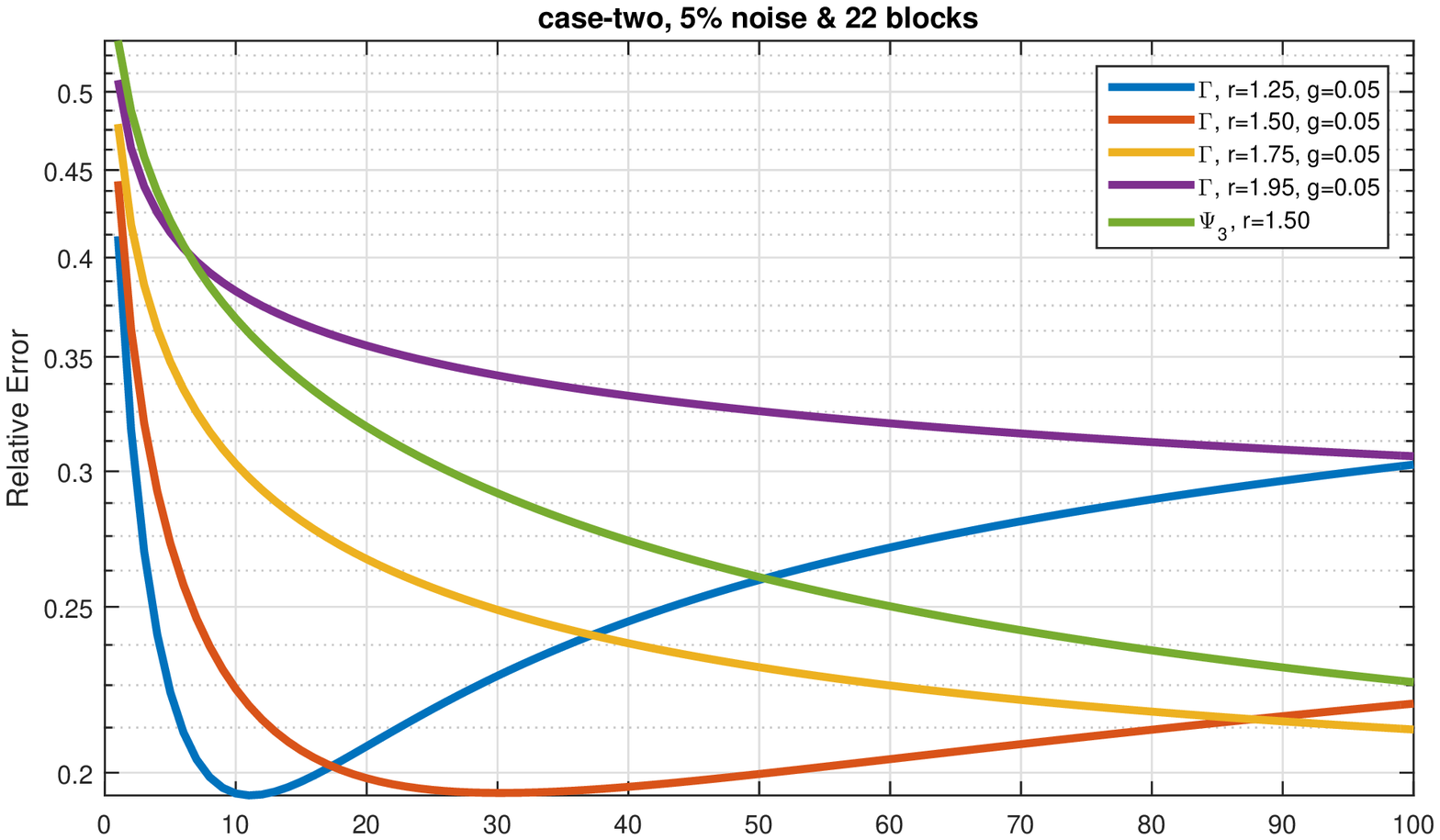}}
  \caption{The effect of using different values for $"r"$ on
$\Gamma$ rule.}
\end{figure}
In Figure \ref{r-fig-iter}
the effect of picking different $r-$values is
displayed. As seen, the choice of proper value of $r$ can have efficacious
impact
on the rate of convergence. The value of $r$ should also be chosen in such a
way that the property (\ref{lambda:less:sigma}) is satisfied, at least
after a few iterations. To insure this, we  take
 $r=1.5$ and $1.75$ for $2\%$ and $5\%$ noise, respectively.
Then (\ref{lambda:less:sigma}) will be satisfied, for $k \geq 2$.
We further remark that
our computational experience shows that the unregularized problem
$(\alpha=0)$ gives results that are indistinguishable from those of
the regularized problem with $\alpha=\underline{\sigma}^2$.
In the sequel we study relative error curves during
$cmax=100$ cycles, but will also extend to $cmax=500$ cycles to study
semiconvergence. Figure \ref{cons}  illustrates the importance of incorporating
constraints, especially for highly underdetermined systems,
during the iterations.

\begin{figure}  
  \centering \label{cons}
  \subfigure[]{\includegraphics[width=0.45\textwidth,height=4cm,
trim={1cm .75cm  1cm  0.5cm },clip]{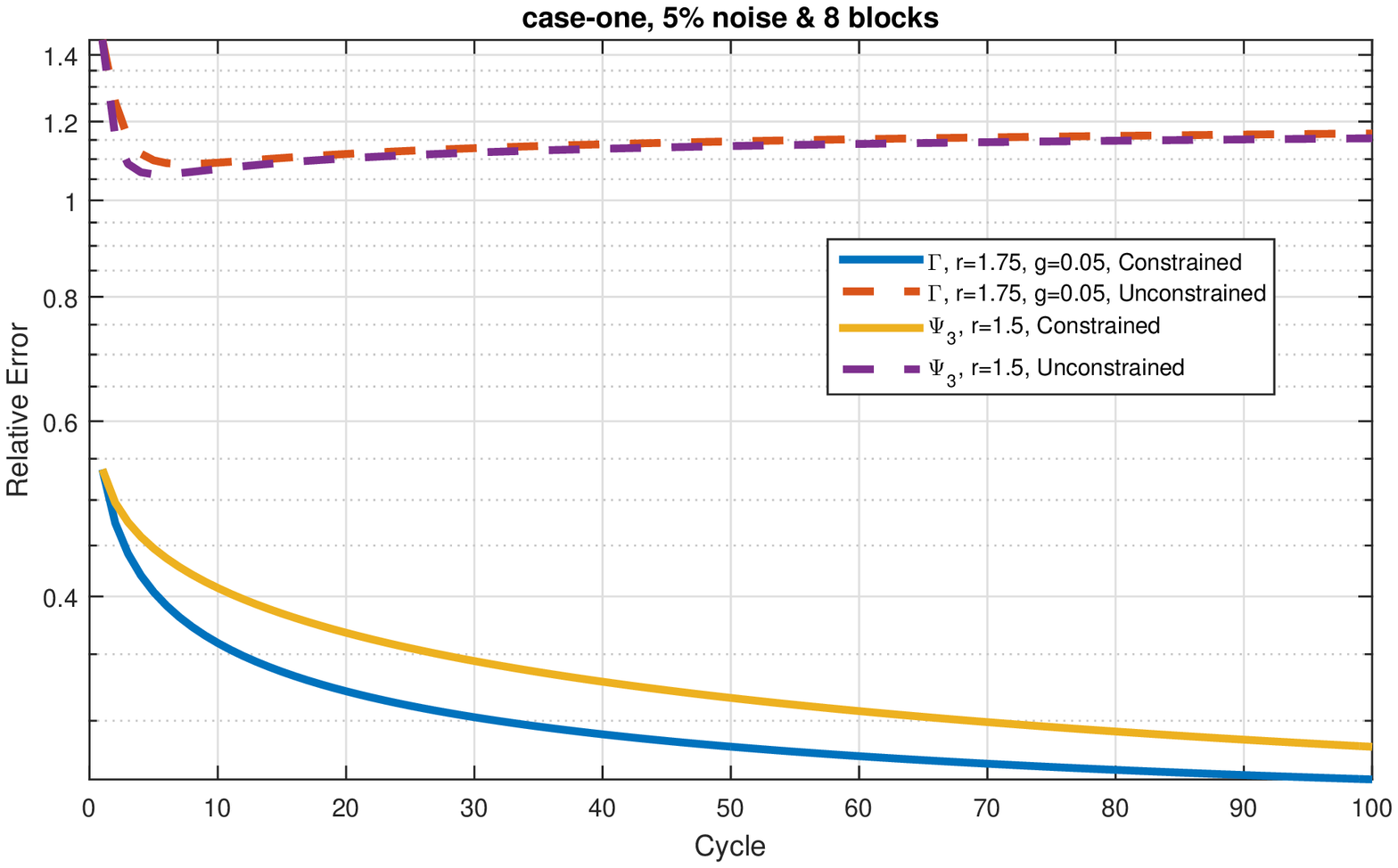}}
  \quad
  \subfigure[]{\includegraphics[width=0.45\textwidth,height=4cm,
trim={1cm .75cm  1cm  0.5cm },clip]{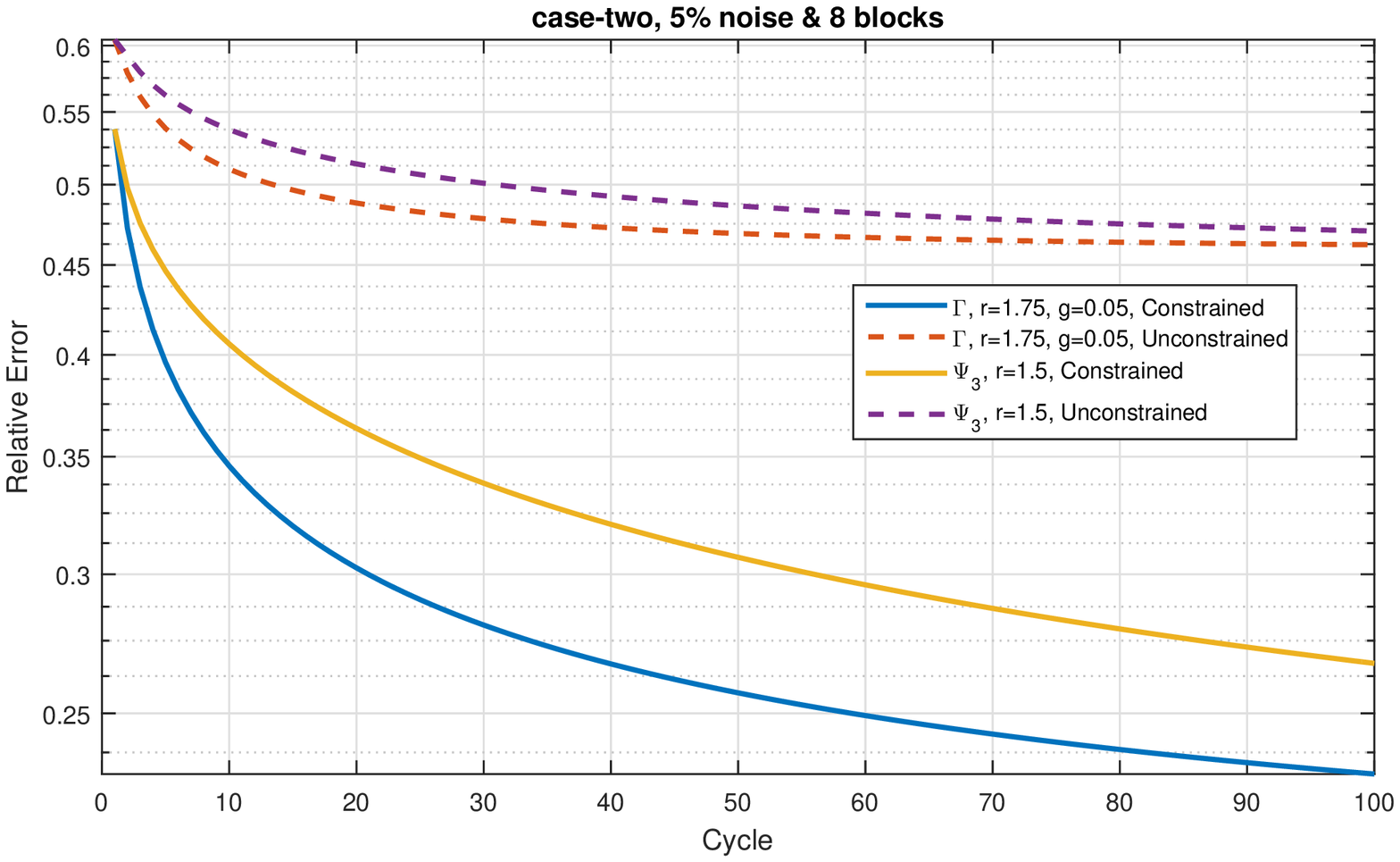}}
  \caption{The
effect of using constraints (solid curves) and not using constraints (dotted
curves) in $\Gamma$ and $\Psi_3$ rules.}
\end{figure}

We next discuss Figures \ref{c1-2n} and \ref{figC2n2}
where we display error-curves, and
 relaxation parameters for both case-one and case-two using
2$\%$ noise-level.
Since the relaxation
parameters showed a similar behavior for the same noise-level we only display
these for case-one.
We observe that the relative error is much smaller for the $\Gamma-$rule
than for the best $\Psi-$rule ($\Psi_3$). The reason for this behavior could be that the corresponding relaxation
parameters are bigger using the $\Gamma-$rule than using the $\Psi-$rule.
\begin{figure} 
    \centering \label{c1-2n}
    \subfigure[]{\includegraphics[width=0.45\textwidth,height=4cm, trim={1cm .75cm  1cm  0.5cm },clip]
{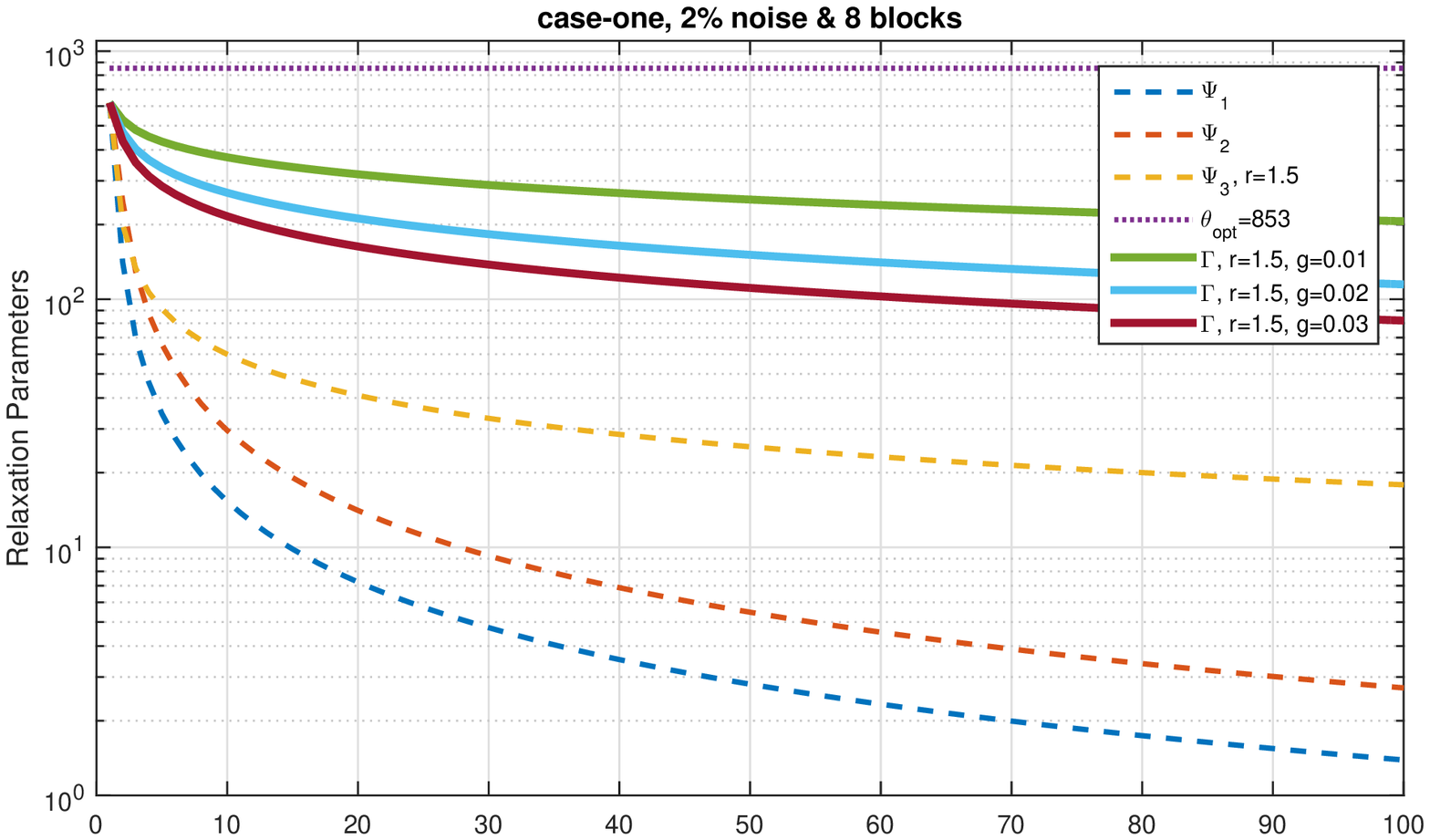}}
    \quad
    \subfigure[]{\includegraphics[width=0.45\textwidth,height=4cm, trim={1cm .75cm  1cm  0.5cm },clip]
{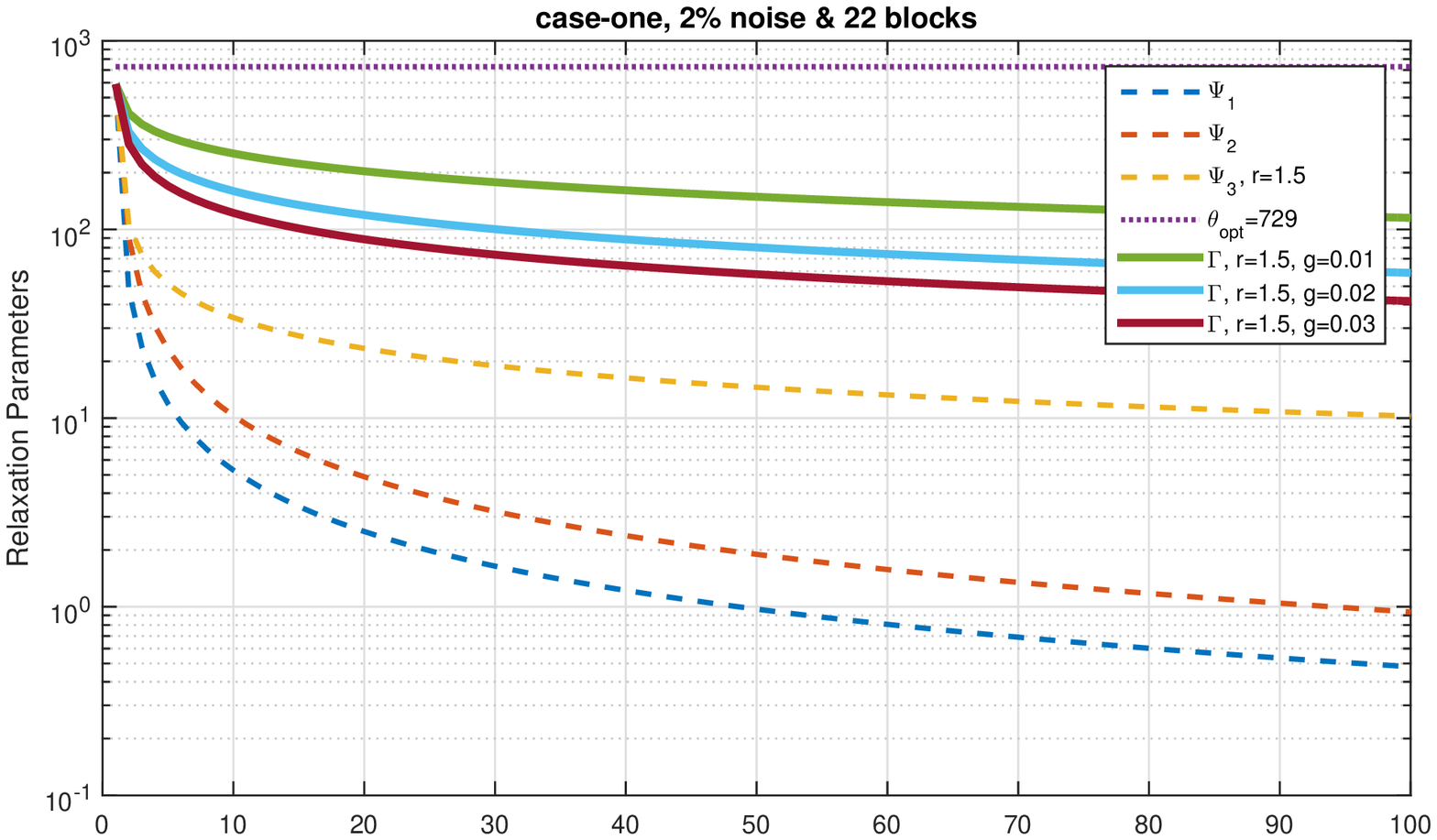}}
    \\
        \subfigure[]{\includegraphics
[width=0.45\textwidth,height=4cm, trim={1cm .75cm  1cm  0.5cm},clip]
{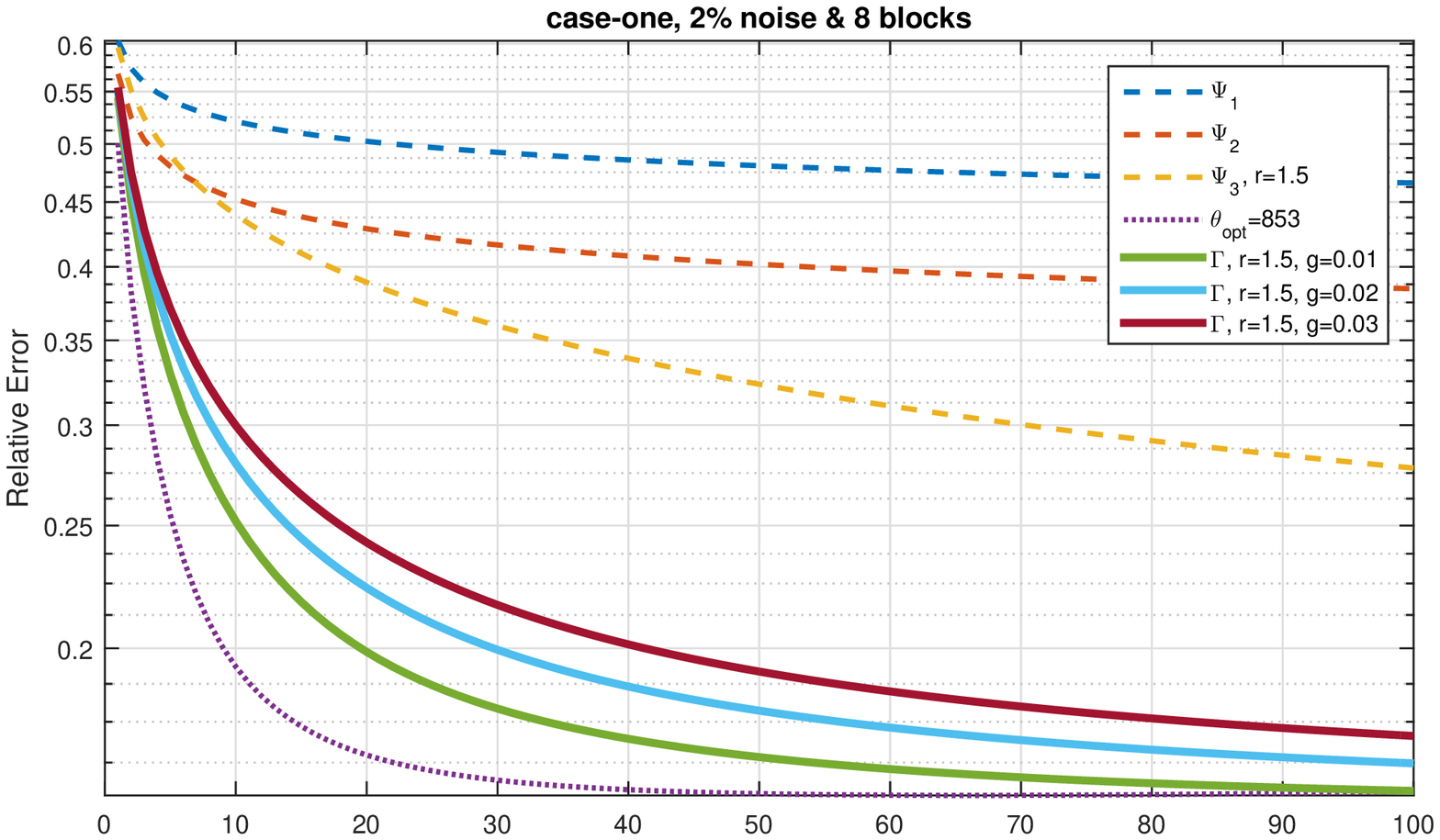}}\label{fig2c}
    \quad
    \subfigure[]{\includegraphics[width=0.45\textwidth,height=4cm, trim={1cm .75cm  1cm  0.3cm},clip]
{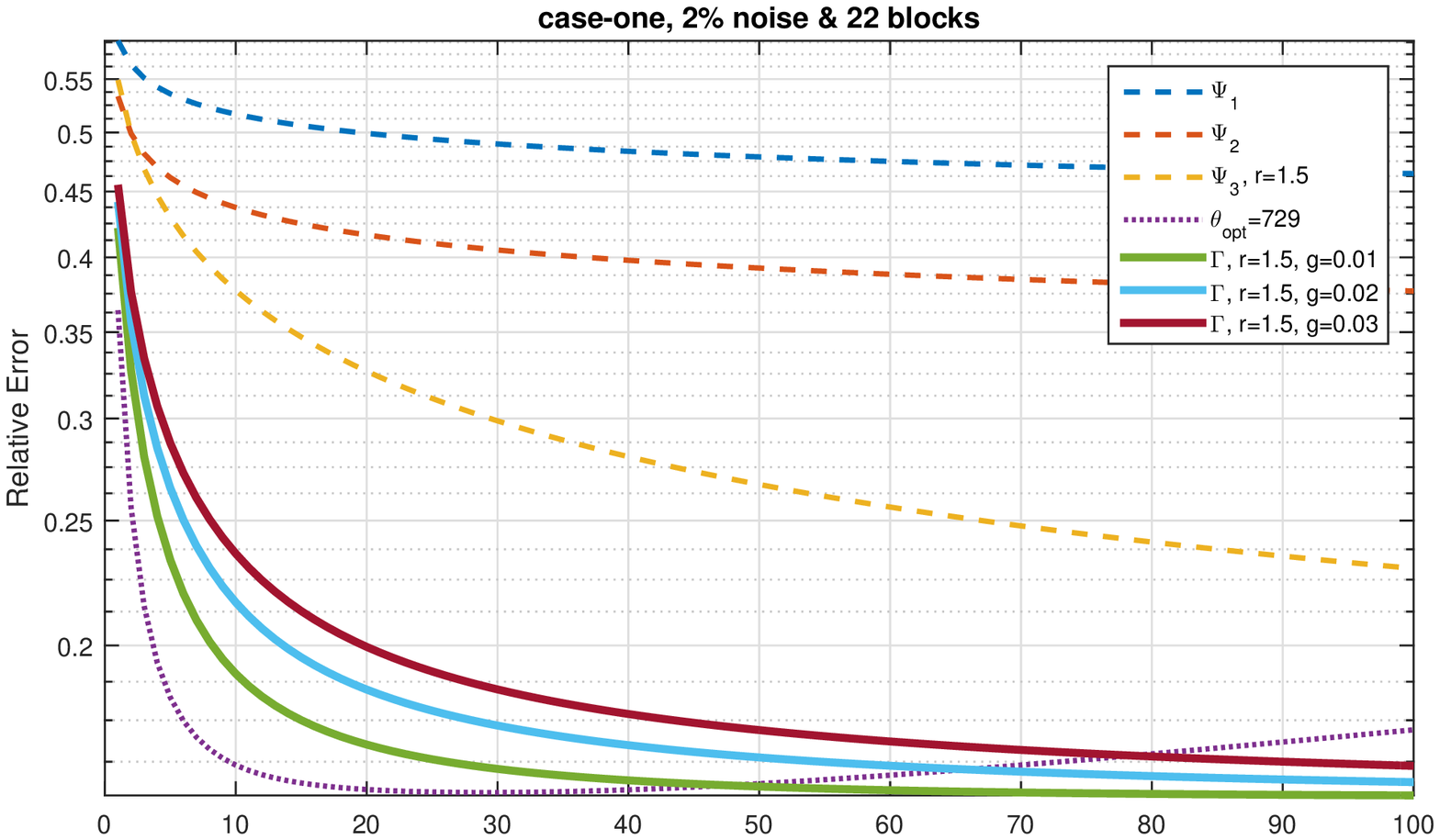}}
    \caption{Case-one with $2\%$ noise.
Relaxation parameter behavior (first row) and relative error history (second row) for $8$ blocks (left column) and $22$ blocks (right column).}
\end{figure}

\begin{figure} 
    \centering
        \subfigure[]{\includegraphics[width=0.45\textwidth,height=4cm,
trim={1cm .75cm  1cm  0.5cm},clip]
{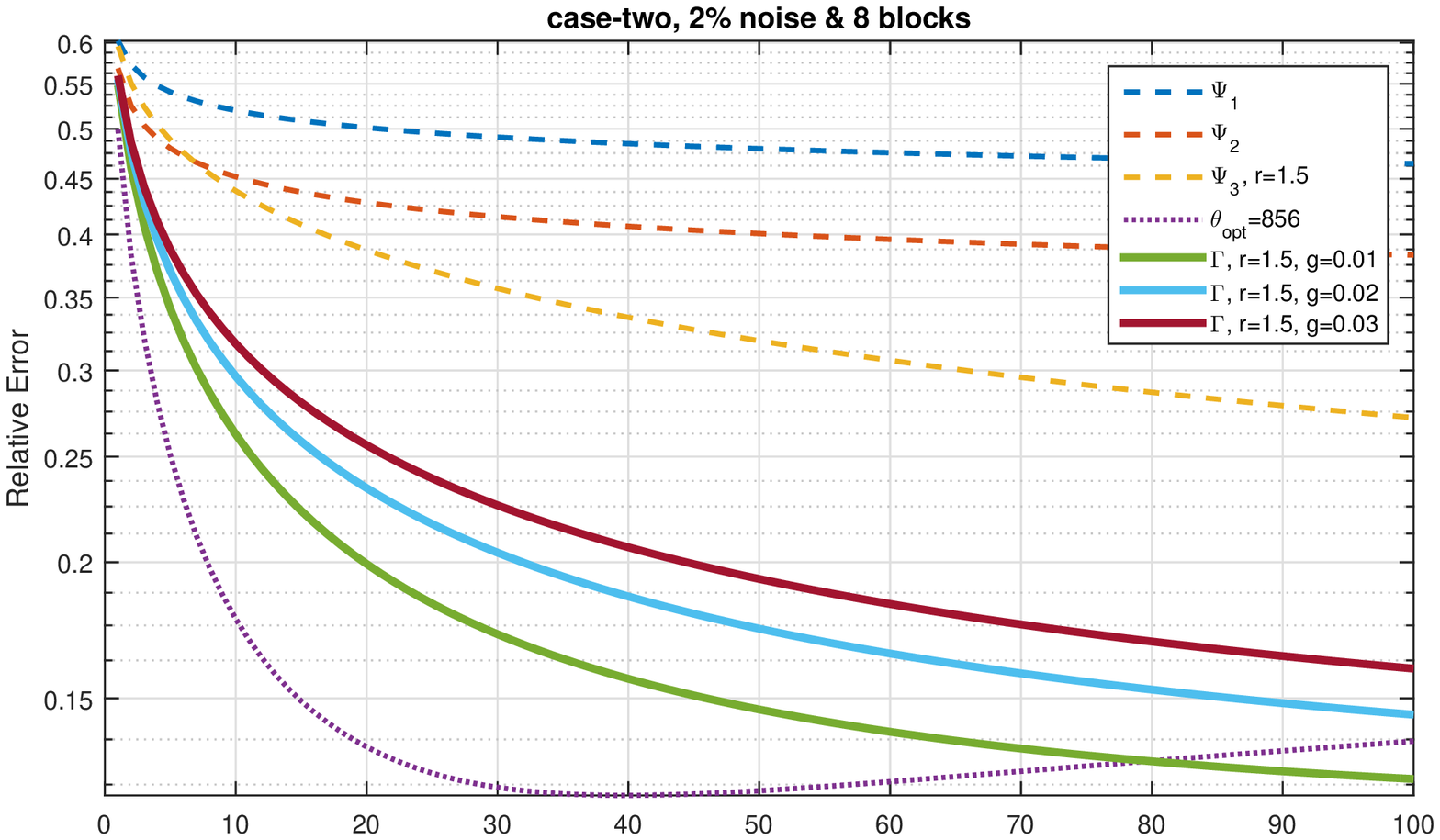}}
    \quad
    \subfigure[]{\includegraphics[width=0.45\textwidth,height=4cm,
trim={1cm .75cm  1cm  0.5cm},clip]
{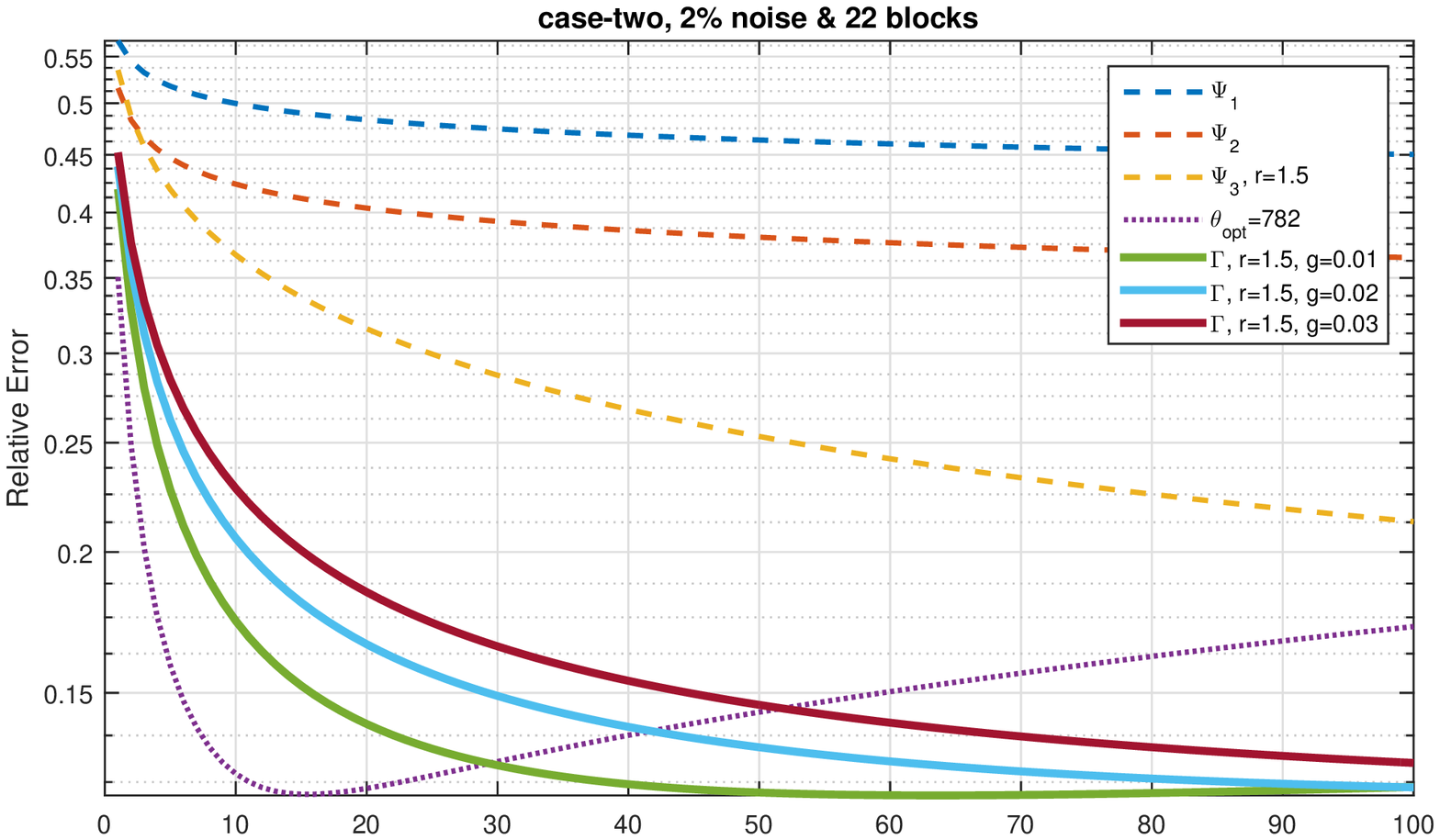}}
    \caption{Case-two with $2\%$ noise.
 Relative error history for $8$ blocks (left column) and $22$ blocks
 (right column).}
\label{figC2n2}
\end{figure}

In Figures \ref{c1-5n} and \ref{figC2n5} 
we show  error curves
and relaxation parameters (for case-two)  using 5$\%$ noise.
It is seen that now 
there is smaller difference in relative error
between the $\Gamma$-rule and the $\Psi_3$-rule.
We see from the figures  that  the
relaxation parameters using the two rules are quite close in value.
We have also listed the minimal error and corresponding cycle number in
Tables \ref{table3new} and \ref{table4new}.
\begin{figure} 
    \centering \label{c1-5n}
        \subfigure[]{\includegraphics[width=0.45\textwidth,height=4cm, trim={1cm .75cm  1cm  0.5cm},clip]
{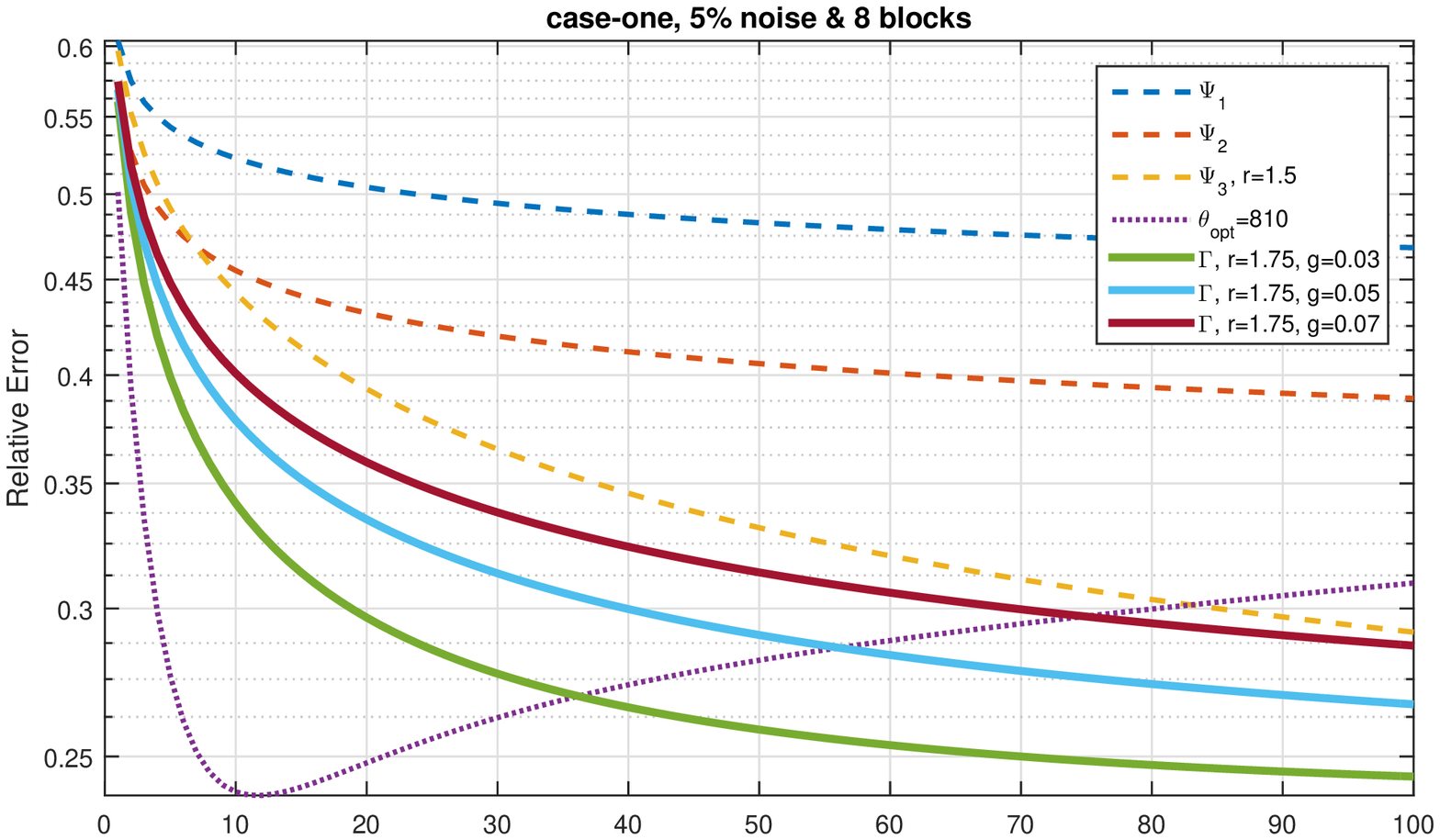}}
    \quad
    \subfigure[]{\includegraphics[width=0.45\textwidth,height=4cm, trim={1cm .75cm  1cm  0.5cm},clip]
{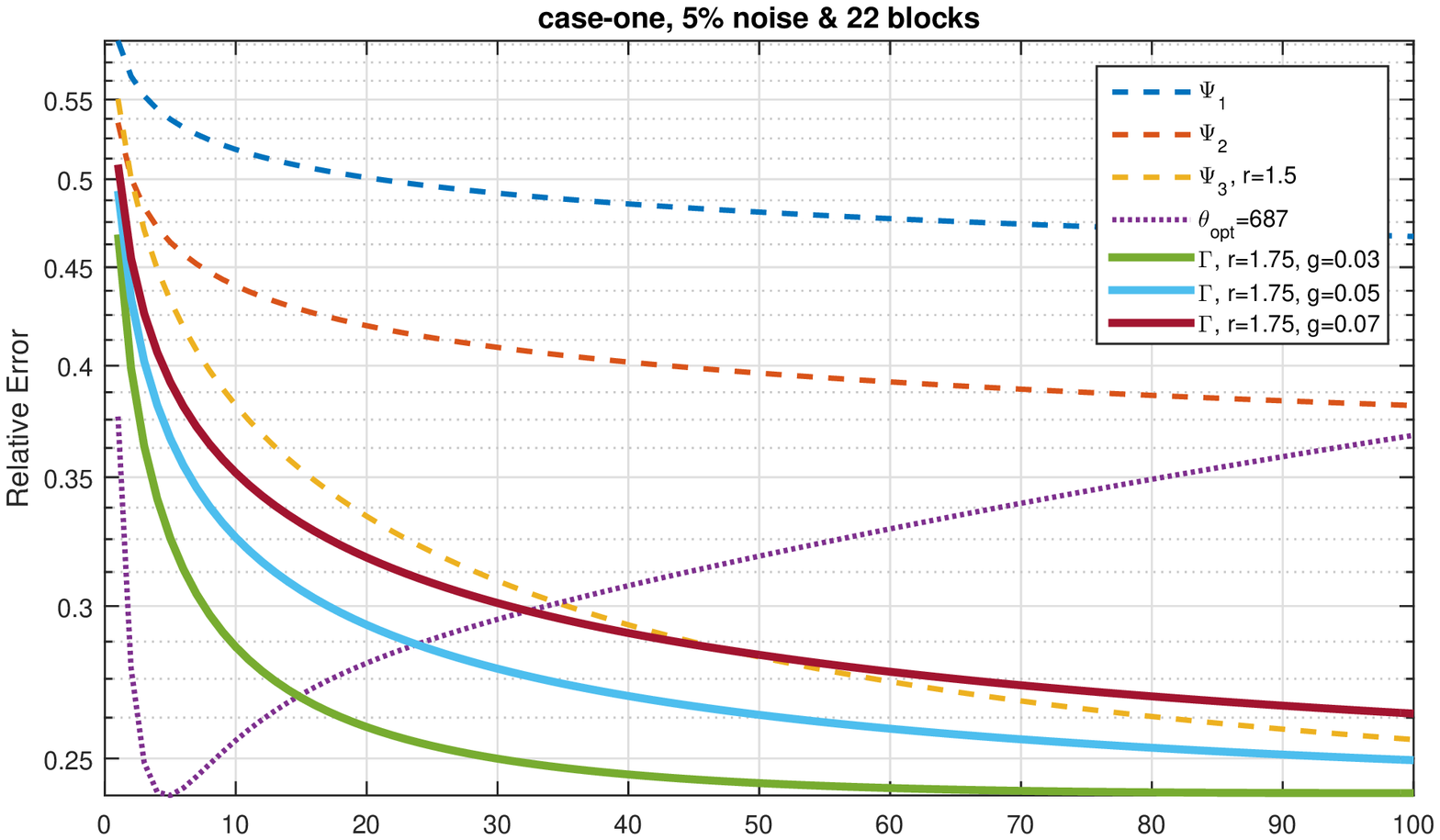}}
    \caption{Case-one with $5\%$ noise.
Relative  error history for $8$ blocks (left column) and $22$ blocks
(right column).}
\end{figure}

\begin{figure}  
    \centering
    \subfigure[]{\includegraphics[width=0.45\textwidth,height=4cm, trim={1cm .75cm  1cm  0.5cm },clip]{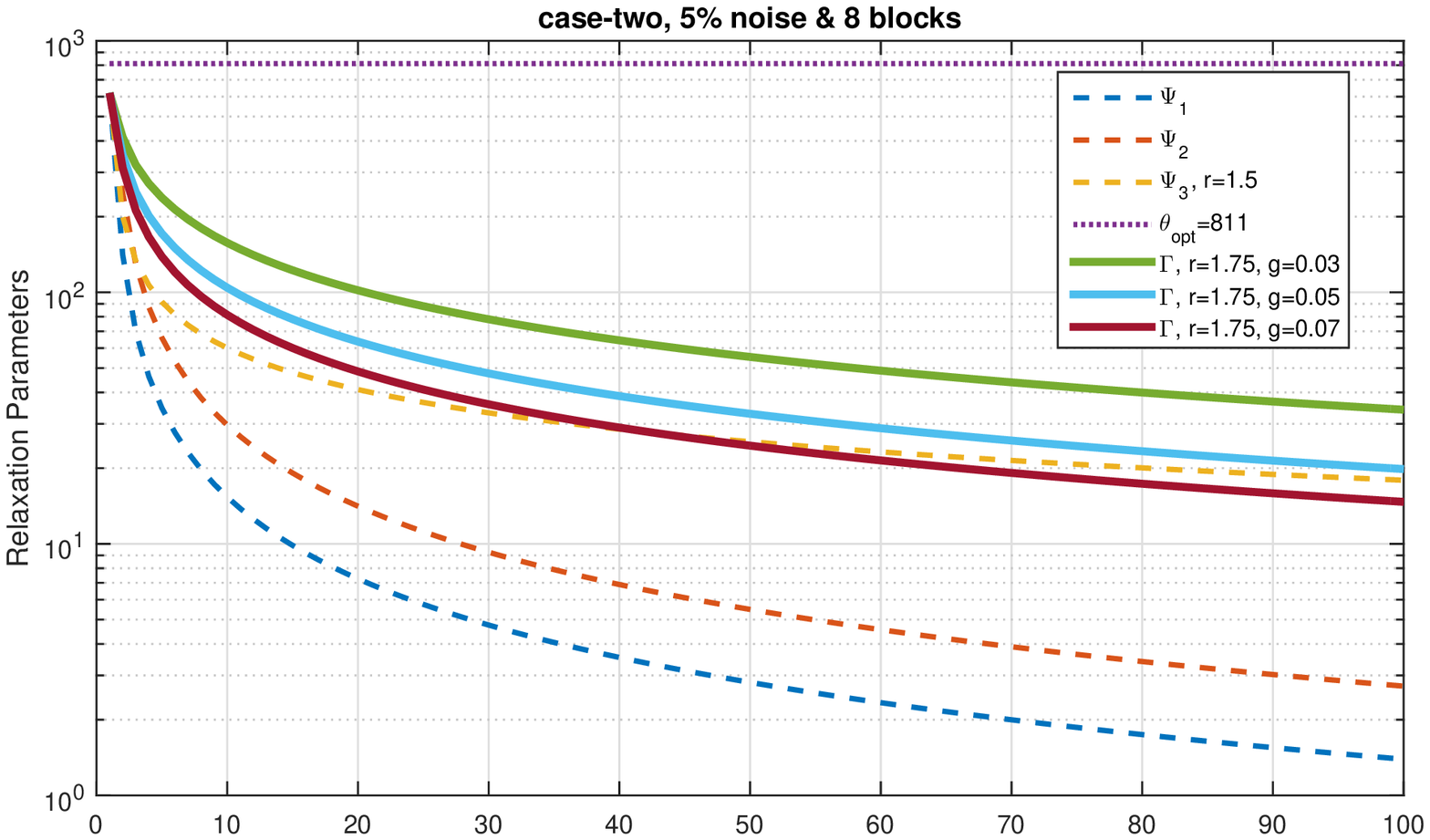}}
    \quad
    \subfigure[]{\includegraphics[width=0.45\textwidth,height=4cm, trim={1cm .75cm  1cm  0.5cm },clip]{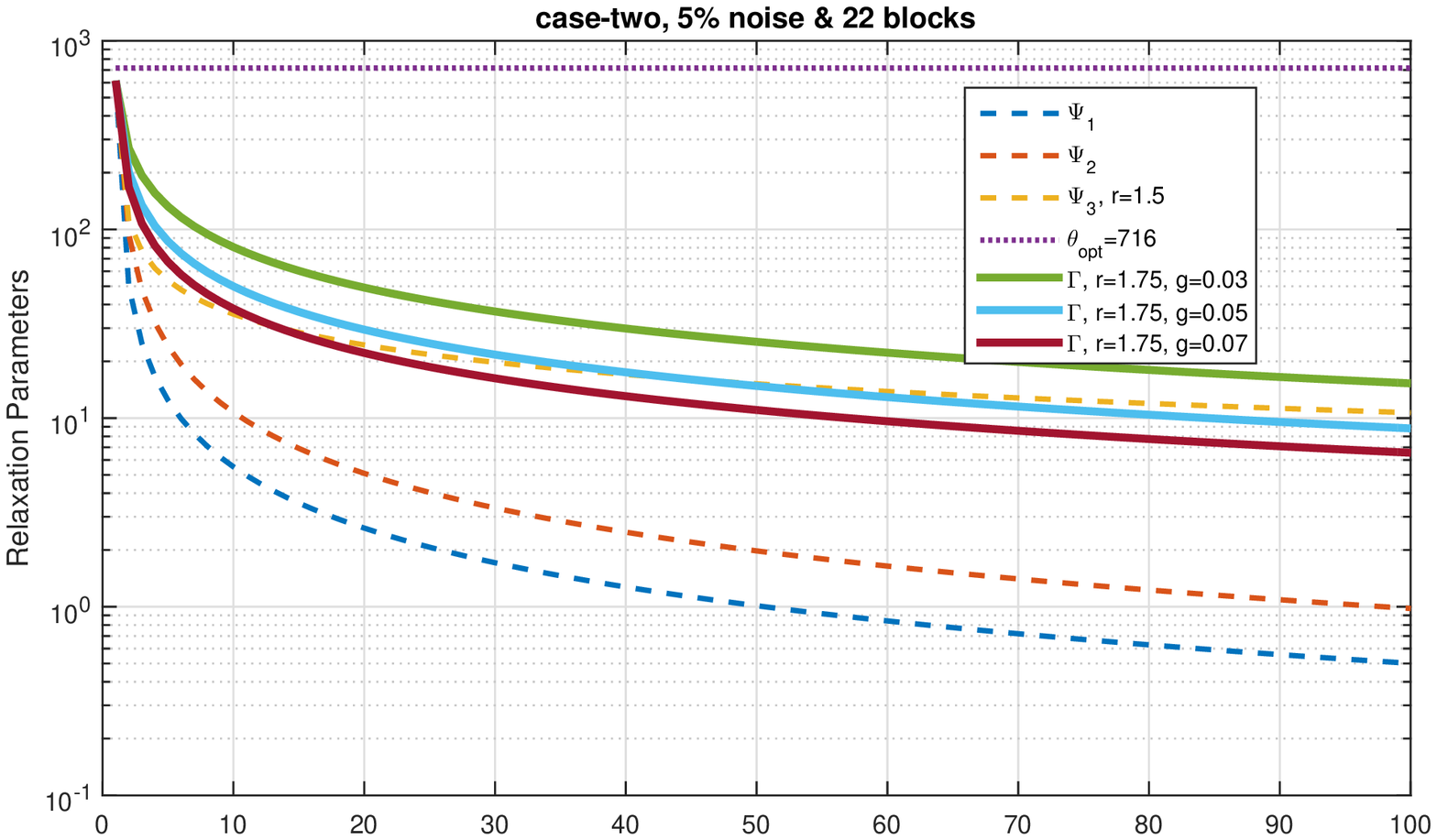}}
    \\

    \subfigure[]{\includegraphics[width=0.45\textwidth,height=4cm, trim={1cm .75cm  1cm  0.5cm},clip]{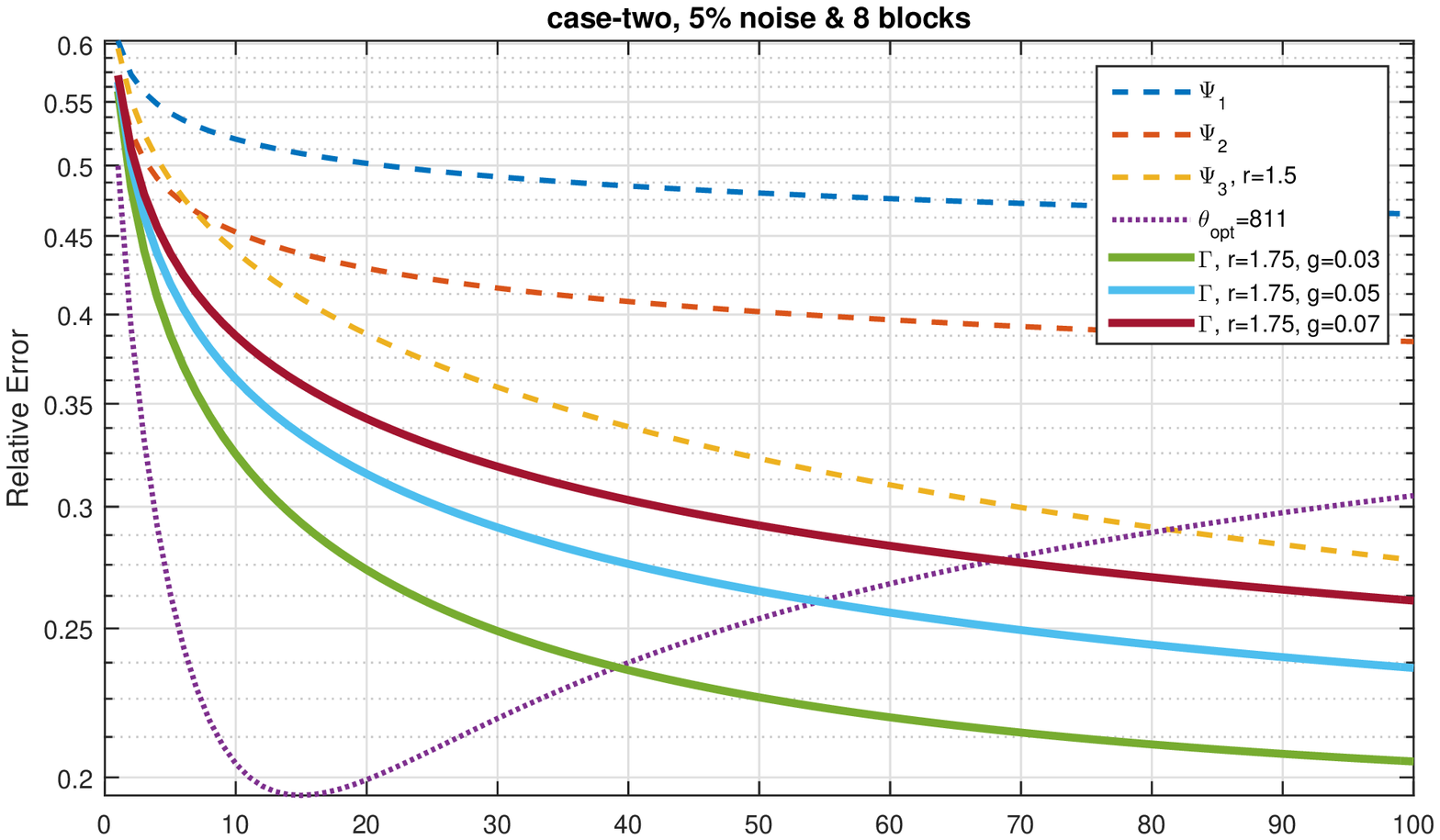}}
    \quad
    \subfigure[]{\includegraphics[width=0.45\textwidth,height=4cm, trim={1cm .75cm  1cm  0.5cm},clip]{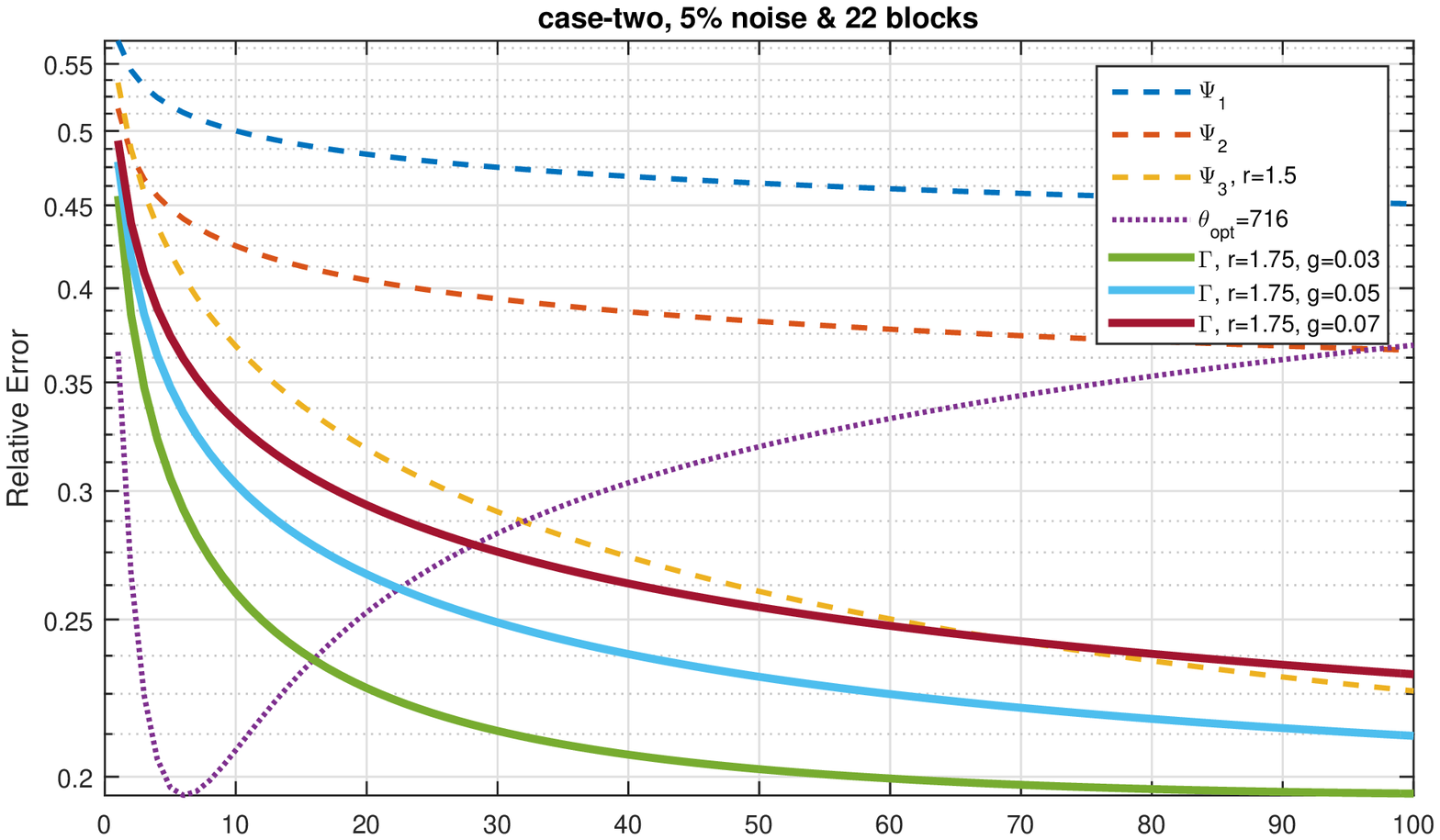}}
    \caption{Case-Two with $5\%$ noise.
 Relaxation parameter behavior (first row) and relative error history (second row) for $8$ blocks (left column) and $22$ blocks (right column).}\label{figC2n5}
\end{figure}

\begin{table}[]
{\small
\caption{Minimum relative error and corresponding cycle number , when $2\%$ noise applied.}
\begin{tabular}{|c|c|c|c|c|c|c|}
\hline
\multirow{3}{*}{\textbf{Test Problem}} & \multirow{3}{*}{\textbf{\# of blocks}} & \multicolumn{5}{c|}{\textbf{Strategy}}                                                                                                                                \\ \cline{3-7}
                                       &                                        & \multirow{2}{*}{\textbf{$\theta_{opt}$}} & \multirow{2}{*}{$\Psi_3$} & \multicolumn{3}{c|}{$\Gamma$} \\ \cline{5-7}
                                       &                                        &                                                          &                                                      & \textbf{$g=0.01$} & \textbf{$g=0.02$} & \textbf{$g=0.03$} \\ \hline
\multirow{2}{*}{\textbf{Case One}}     & \textbf{8}                             & (66, 0.1531)                                             & (100, 0.2914)                                         & (100, 0.1543)   & (100, 0.1622)   & (100, 0.1706)   \\ \cline{2-7}
                                       & \textbf{22}                            & (29, 0.1538)                                              & (100, 0.2295)                               & (100, 0.1530)   & (100, 0.1567)   & (100, 0.1613)   \\ \hline
\multirow{2}{*}{\textbf{Case Two}}     & \textbf{8}                             & (40, 0.1221)                                             & (100, 0.2715)                                        & (100, 0.1265)   & (100, 0.1449)   & (100, 0.1597)   \\ \cline{2-7}
                                       & \textbf{22}                            & (15, 0.1219)                                             & (100, 0.2128)                                        & (64, 0.1217)    & (100, 0.1237)   & (100, 0.1300)   \\ \hline
\end{tabular}
\label{table3new}
}
\end{table}

\begin{table}[h!]
{\small
\caption{Minimum relative error and corresponding cycle number , when $5\%$ noise applied.}
\begin{tabular}{|c|c|c|c|c|c|c|}
\hline
\multirow{3}{*}{\textbf{Test Problem}} & \multirow{3}{*}{\textbf{\# of blocks}} & \multicolumn{5}{c|}{\textbf{Strategy}}                                                                                                                                \\ \cline{3-7}
                                       &                                        & \multirow{2}{*}{\textbf{$\theta_{opt}$}} & \multirow{2}{*}{$\Psi_3$} & \multicolumn{3}{c|}{$\Gamma$} \\ \cline{5-7}
                                       &                                        &                                                          &                                                      & $g=0.03$          & $g=0.05$          & $g=0.07$          \\ \hline
\multirow{2}{*}{\textbf{Case One}}     & \textbf{8}                             & (12, 0.2383)                                             & (100, 0.2914)                                         & (100, 0.2439)   & (100, 0.2666)   & (100, 0.2866)   \\ \cline{2-7}
                                       & \textbf{22}                            & (5, 0.2392)                                              & (100, 0.2557)                                           & (97, 0.2398)    & (100, 0.2495)   & (100, 0.2639)   \\ \hline
\multirow{2}{*}{\textbf{Case Two}}              & \textbf{8}                             & (15, 0.1947)                                    & (100, 0.2769)                                           & (100, 0.2606)   & (100, 0.2356)   & (100, 0.2408)   \\ \cline{2-7}
                                       & \textbf{22}                            & (6, 0.1948)                                              & (100, 0.2559)                                        & (100, 0.1952)   & (100, 0.2200)   & (100, 0.2313)   \\ \hline
\end{tabular}
\label{table4new}
}
\end{table}

We note from Figures \ref{c1-2n}, \ref{figC2n2}, \ref{c1-5n} and \ref{c1-5n} 
that semi-convergence has hardly begun
for the $\Gamma$-rule. To see more clearly this effect we extend in
Figure \ref{cmax500marked}
the error curves (for two cases) to $cmax=500$.
In this Figure
we have marked the point where the error has its minimum,
i.e. the cycle number where semi-convergence starts.
We see that the $\theta-$opt rule has the fastest convergence
(note again that we
used the exact phantom for training). For this rule the semi-convergence
behavior is quite pronounced, and hence it requires a reliable stopping
criterion
(the choice of stopping criterion is an interesting issue but is not addressed
in this paper). In contrast the $\Psi-$ and $\Gamma-$ rules show little
effect of semi-convergence, and therefore it is less critical
where the iterations are stopped.
\begin{figure}  
  \centering \label{cmax500marked}
  \subfigure[]{\includegraphics[width=0.47\textwidth,height=4cm,trim={1cm .75cm  1cm  0.5cm },clip]{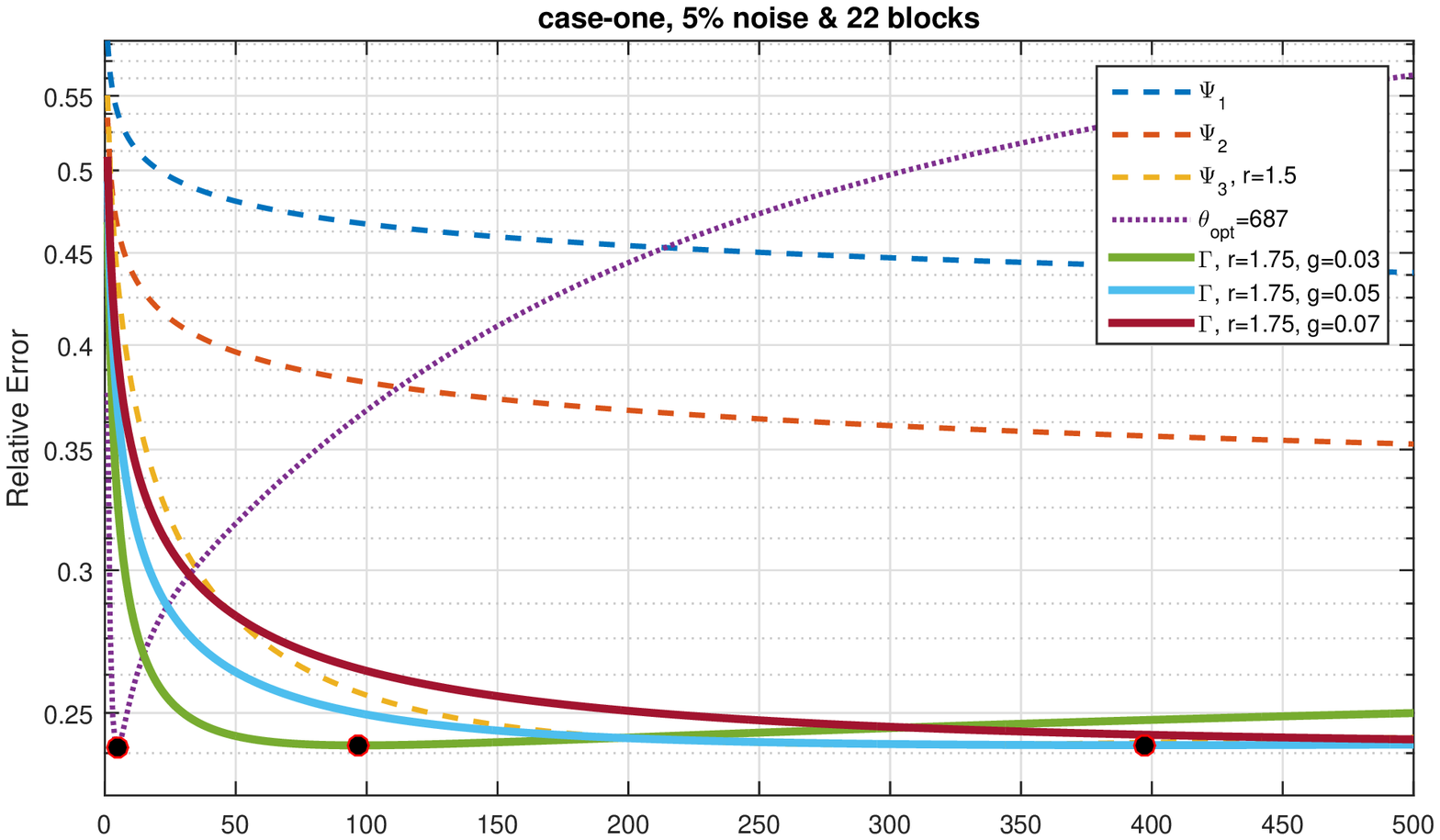}}
  \quad
  \subfigure[]{\includegraphics[width=0.47\textwidth,height=4cm, trim={1cm .75cm  1cm  0.5cm },clip]{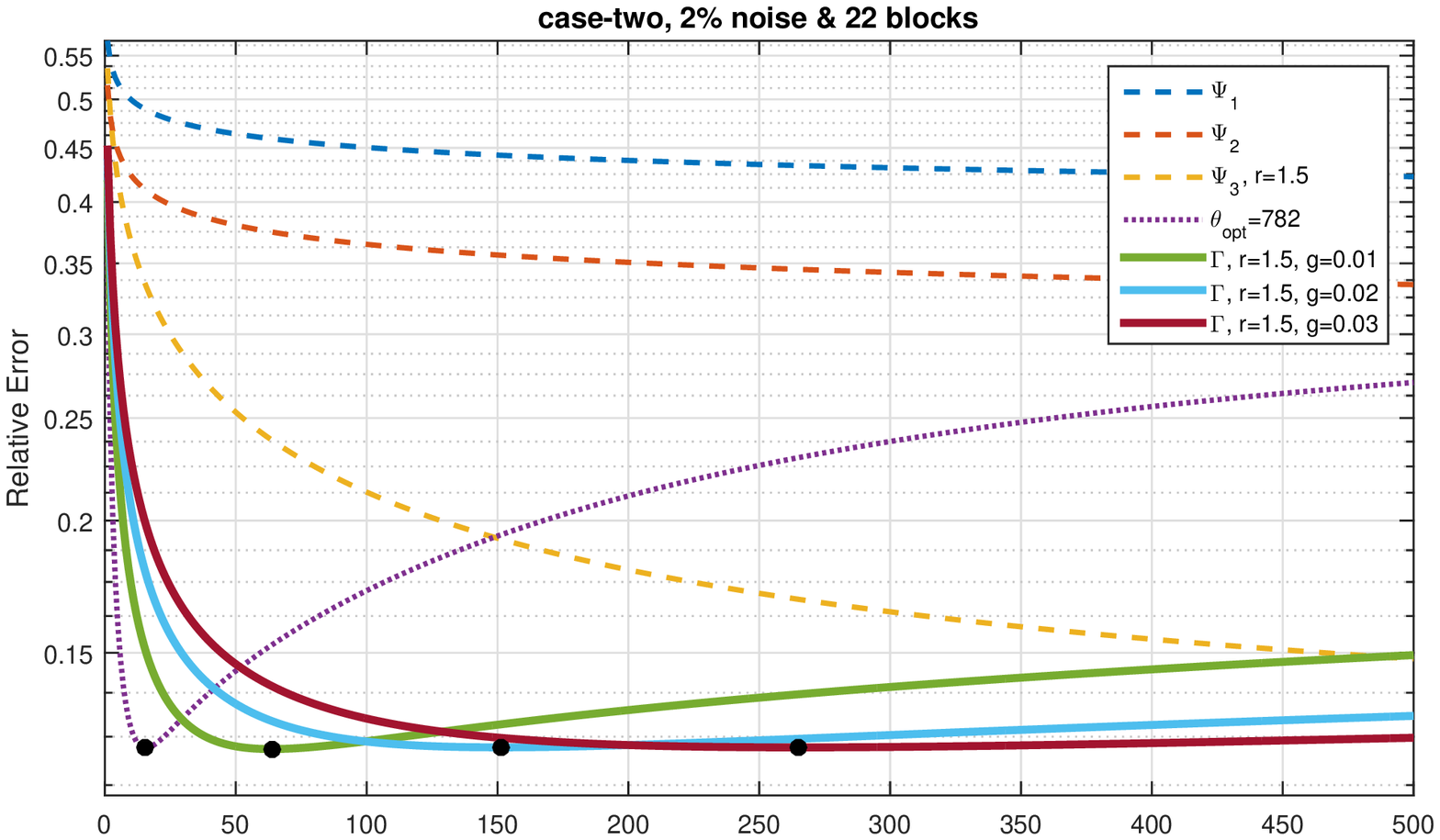}}
  \caption{The symbol $\bullet$ indicates the point that semi-convergence starts.}
\end{figure}

The behavior of the relative noise error
$\|x^k-\bar{x}^k\|/\|\bar{x}^k\|$, and the relative iteration error
$\|\bar{x}^k-x^*\|/\|x^*\|$ for case-one is shown in Figure \ref{upper:fig}.
Figure \ref{fig:recon}
shows the phantom and reconstructions using
the $\Psi_3$ and the $\Gamma$-rule respectively for case-one. To better judge
the quality of the
two reconstructions we display also
the corresponding difference images.
These are defined as the difference between the phantom and the respective
reconstruction. More artifacts and noise can be seen in the left image
($\Psi_3$)
than in the right image ($\Gamma$).

\begin{figure}  
  \centering
\includegraphics[width=0.65\textwidth,trim={1cm .75cm  1cm  0.25cm },clip]{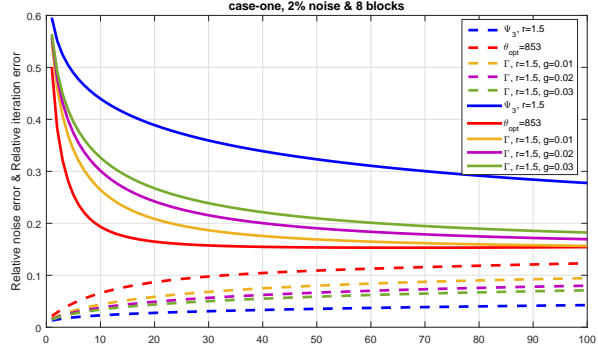}
\caption{Relative noise error $\|x^{k}-\bar{x}^k\|/\|x^{\ast}\|$ (dashed), and the relative iteration error $\|\bar{x}^{k}-x^{\ast}\|/\|x^{\ast}\|$.}\label{upper:fig}
\end{figure}

\begin{figure}[h!]
  \centering \label{fig:recon}
  \subfigure[]{\includegraphics[width=0.65\textwidth,
trim={1.25cm 3.35cm  1.25cm  1.5cm },clip]{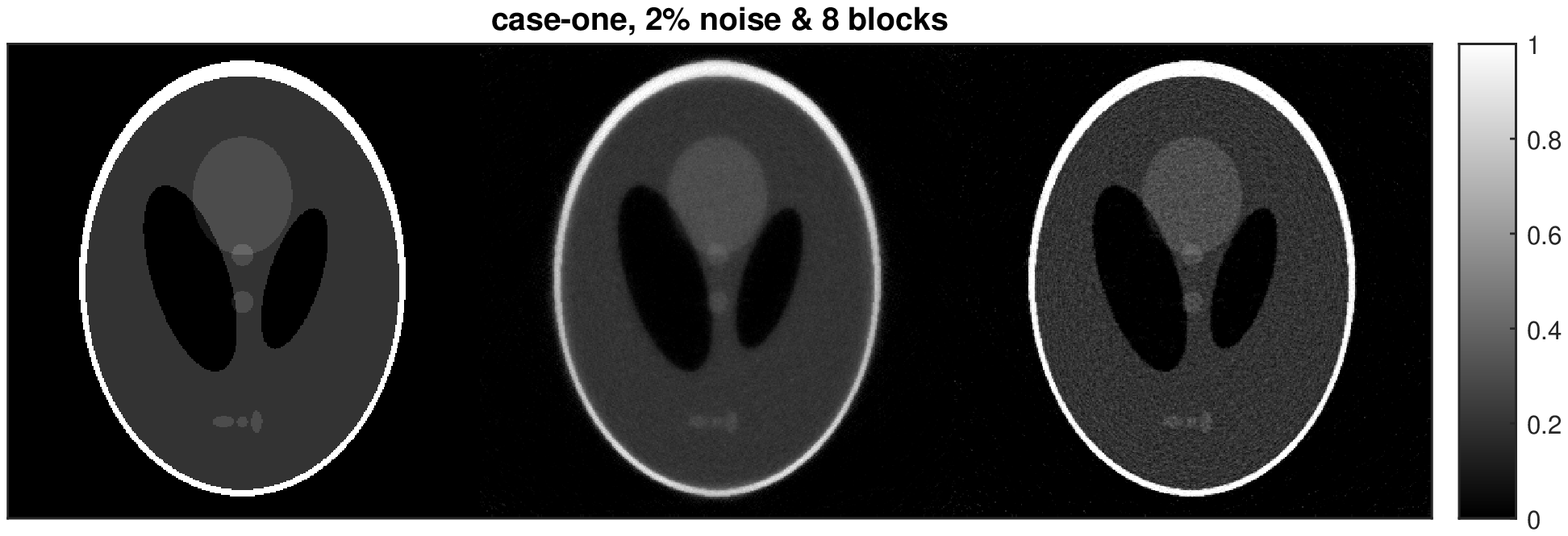}}
  \\
  \subfigure[]{\includegraphics[width=0.65\textwidth,
trim={1.25cm 2cm  1.5cm  1cm },clip]{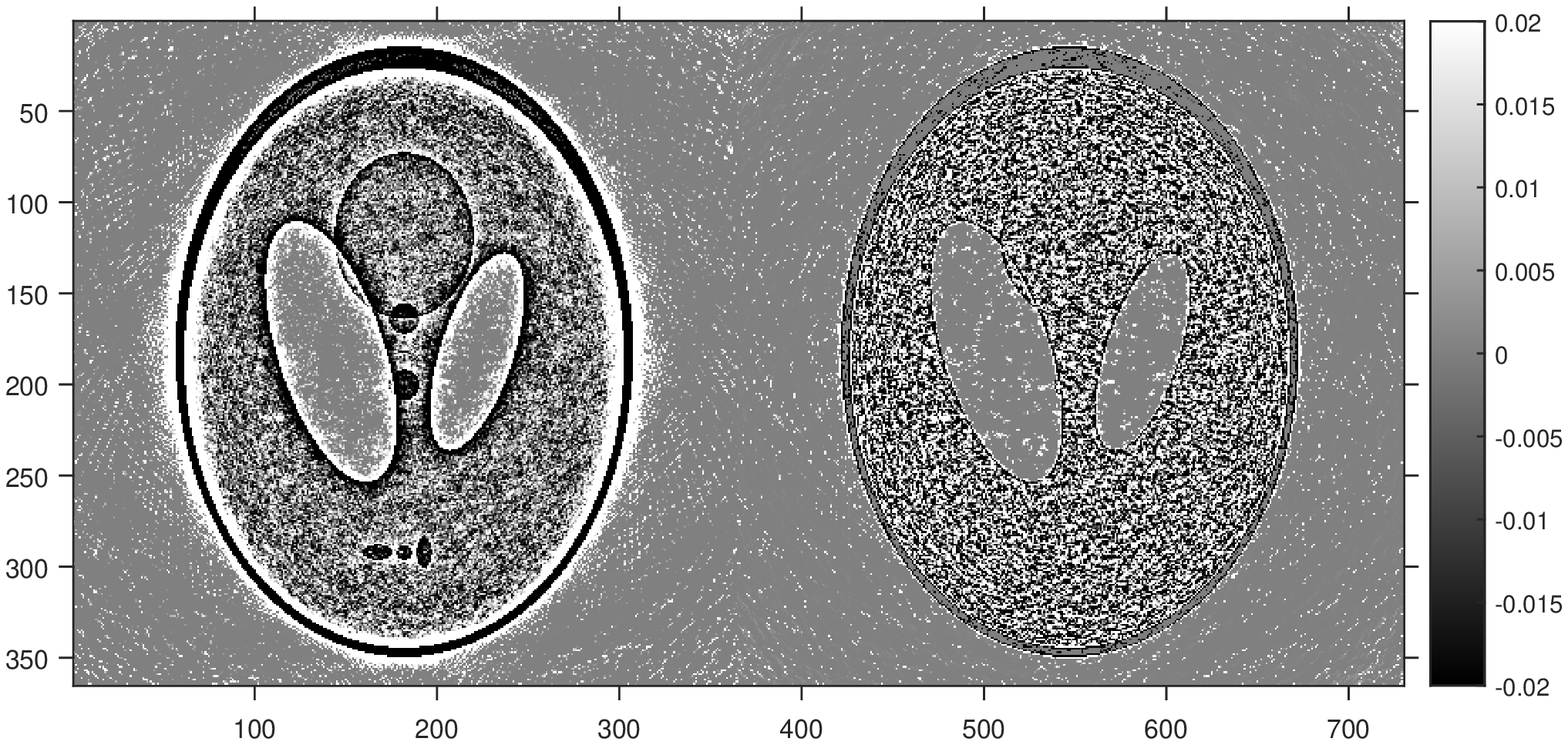}}
  \caption{(a) The original image (left), the reconstructed image using $\Psi_3$ strategy (middle) and  the reconstructed image using the $\Gamma$ strategy, and (b) the difference images for the $\Psi_3$ strategy (left) and $\Gamma$ strategy (right).}
\end{figure}

\section{Conclusion}\label{Conclusion}
We  define a sequential block-iterative iteration (\ref{e-xkac}) for
solving split feasibility problems in Hilbert space.
In this respect it compliments the
simultaneous block iteration given in \cite{CE-2005}
(defined in the finite dimensional case).
The basic operators involved are weakly regular cutters which includes e.g.  metric projections.
A complete convergence analysis is provided.
When the projecting sets are polyhedral it is also shown that the
iterates converge  linearly.
 The projected Block Iterative method (P-BIM) Algorithm \ref{BIP:algt} is a special case
 of our general method.
We also consider the noise error of P-BIM, and derive  a new upper
 bound.
This bound generalizes earlier bounds given in
\cite{EHN2012,elfving2010,Elfving2014,NK2015}.
In particular we extend the noise error analysis of block-iteration
to {\it projected} block-iteration, and also to the case
 when the relaxation
parameters are allowed to depend  on the noise.
Based on the new bound a new rule for picking relaxation parameters is derived.
 We demonstrate the performance of P-BIM using this and other relaxation
parameter rules on examples taken from tomographic imaging.

\section*{Acknowledgement}
We thank two anonymous referees for useful suggestions, and professor Per
Christian Hansen and professor Touraj Nikazad for their comments on an earlier
version.

\bibliographystyle{siamplain}

\bibliography{refMahdi}


\end{document}

%% file: ex_shared.tex
\usepackage{amsmath}
\usepackage{amssymb}
\usepackage{amsfonts}
\usepackage{lipsum}
\usepackage{graphicx}
\usepackage{subfigure}
\usepackage{epstopdf}
\usepackage{algorithm}
\usepackage[utf8]{inputenc}
\usepackage[T1]{fontenc}
\usepackage{url}
\usepackage{algpseudocode}
\usepackage{multirow}

\ifpdf
  \DeclareGraphicsExtensions{.eps,.pdf,.png,.jpg}
\else
  \DeclareGraphicsExtensions{.eps}
\fi

\newsiamremark{remark}{Remark}
\newsiamremark{hypothesis}{Hypothesis}
\crefname{hypothesis}{Hypothesis}{Hypotheses}
\newsiamthm{claim}{Claim}

\headers{Convergence and Semi-convergence of a Class P-BIM}{M. Mirzapour, A. Cegielski \& T. Elfving}

\title{Convergence and Semi-convergence of a Class of Constrained Block Iterative Methods}

\author{Mahdi Mirzapour\thanks{Department of Mathematics, Faculty of Sciences, Bu-Ali Sina University, Hamedan, Iran
  (\email{m.mirzapour@basu.ac.ir})}\and Andrzej Cegielski\thanks
{Institute of Mathematics, University of Zielona G\'{o}ra, Zielona G\'{o}ra, Poland(\email{a.cegielski@wmie.uz.zgora.pl})} \and Tommy Elfving\thanks{Department of Mathematics, Link\"{o}ping University, SE-581 83 Link\"{o}ping, Sweden
  (\email{tommy.elfving@liu.se}).}}

\usepackage{amsopn}


%% file: Final-black.bbl
\begin{thebibliography}{10}

\bibitem{aharoni1989block}
{\sc R.~Aharoni and Y.~Censor}, {\em Block-iterative projection methods for
  parallel computation of solutions to convex feasibility problems}, Linear
  Algebra and Its Applications, 120 (1989), pp.~165--175.

\bibitem{Bai-Buc}
{\sc Z.-Z. Bai, A.~Buccini, K.~Hayami, L.~Reichel, J.-F. Yin, and N.~Zheng},
  {\em Modulus-based method for constrained {T}ikhonov regularization}, Journal
  of {C}omputational and {A}pplied {M}athematics, 319 (2017), pp.~1--13.

\bibitem{BB96}
{\sc H.~Bauschke and J.~Borwein}, {\em On projection algorithms for solving
  convex feasibility problems}, SIAM review, 38 (1996), pp.~367--426.

\bibitem{BC17}
{\sc H.~H. Bauschke and P.~L. Combettes}, {\em Convex analysis and monotone
  operator theory in {Hilbert} spaces}, vol.~408, CMS Books in Mathematics,
  Springer, Cham, 2017.

\bibitem{bertero1998introduction}
{\sc M.~Bertero and P.~Boccacci}, {\em Introduction to {I}nverse {P}roblems in
  {I}maging}, CRC press, 1998.

\bibitem{Byr02}
{\sc C.~Byrne}, {\em Iterative oblique projection onto convex sets and the
  split feasibility problem}, Inverse problems, 18 (2002), pp.~441--453.

\bibitem{BCGR12}
{\sc C.~Byrne, Y.~Censor, A.~Gibali, and S.~Reich}, {\em The split common null
  point problem}, J. Nonlinear Convex Anal, 13 (2012), pp.~759--775.

\bibitem{byrne1996block}
{\sc C.~L. Byrne}, {\em Block-iterative methods for image reconstruction from
  projections}, IEEE Transactions on Image Processing, 5 (1996), pp.~792--794.

\bibitem{C2013}
{\sc A.~Cegielski}, {\em Iterative {M}ethods for {F}ixed {P}oint {P}roblems in
  {H}ilbert {S}paces}, vol.~2057 of Lecture Notes in Mathematics, Springer,
  Heidelberg, 2012.

\bibitem{Ceg15}
{\sc A.~Cegielski}, {\em General method for solving the split common fixed
  point problem}, Journal of Optimization Theory and Applications, 165 (2015),
  pp.~385--404.

\bibitem{CRZ18}
{\sc A.~Cegielski, S.~Reich, and R.~Zalas}, {\em Regular sequences of
  quasi-nonexpansive operators and their applications}, SIAM Journal on
  Optimization, 28 (2018), pp.~1508--1532.

\bibitem{CRZ20}
{\sc A.~Cegielski, S.~Reich, and R.~Zalas}, {\em Weak, strong and linear
  convergence of the {CQ}-method via the regularity of {L}andweber operators},
  Optimization, 69 (2020), pp.~605--636.

\bibitem{CE-1995}
{\sc Y.~Censor and T.~Elfving}, {\em A multiprojection algorithm using
  {B}regman projections in a product space}, Numerical Algorithms, 8 (1994),
  pp.~221--239.

\bibitem{censor2008diagonally}
{\sc Y.~Censor, T.~Elfving, G.~T. Herman, and T.~Nikazad}, {\em On diagonally
  relaxed orthogonal projection methods}, SIAM Journal on Scientific Computing,
  30 (2008), pp.~473--504.

\bibitem{CE-2005}
{\sc Y.~Censor, T.~Elfving, N.~Kopf, and T.~Bortfeld}, {\em The multiple-sets
  split feasibility problem and its applications for inverse problems}, Inverse
  Problems, 21 (2005), p.~2071.

\bibitem{censor2001component}
{\sc Y.~Censor, D.~Gordon, and R.~Gordon}, {\em Component averaging: An
  efficient iterative parallel algorithm for large and sparse unstructured
  problems}, Parallel computing, 27 (2001), pp.~777--808.

\bibitem{CS09}
{\sc Y.~Censor and A.~Segal}, {\em The split common fixed point problem for
  directed operators}, J. Convex Anal, 16 (2009), pp.~587--600.

\bibitem{CZ2014}
{\sc Y.~Censor and A.~Zaslavski}, {\em String-averaging projected subgradient
  methods for constrained minimization}, Optimization Methods and Software, 29
  (2014), pp.~658--670.

\bibitem{cimmino1938}
{\sc G.~Cimmino}, {\em Cacolo approssimato per le soluzioni dei systemi di
  equazioni lineari}, La Ricerca Scientifica, II 9 (1938), pp.~326--333.

\bibitem{elfving2018row}
{\sc T.~Elfving}, {\em Row and column based iterations}, Applied Analysis and
  Optimization, 2 (2018), pp.~219--236.

\bibitem{EHN2012}
{\sc T.~Elfving, P.~C. Hansen, and T.~Nikazad}, {\em Semiconvergence and
  relaxation parameters for projected {SIRT} algorithms}, SIAM Journal on
  Scientific Computing, 34 (2012), pp.~A2000--A2017.

\bibitem{Elfving2014}
{\sc T.~Elfving, P.~C. Hansen, and T.~Nikazad}, {\em Semi-convergence
  properties of {K}aczmarz’s method}, Inverse Problems, 30 (2014), p.~055007.

\bibitem{EN2009}
{\sc T.~Elfving and T.~Nikazad}, {\em Properties of a class of block-iterative
  methods}, Inverse Problems, 25 (2009), p.~115011.

\bibitem{elfving2010}
{\sc T.~Elfving, T.~Nikazad, and P.~C. Hansen}, {\em Semi-convergence and
  relaxation parameters for a class of {SIRT} algorithms}, Electronic
  Transactions on Numerical Analysis, 37 (2010), pp.~321--336.

\bibitem{engl1996regularization}
{\sc H.~W. Engl, M.~Hanke, and A.~Neubauer}, {\em Regularization of {I}nverse
  {P}roblems}, vol.~375, Springer Science \& Business Media, 1996.

\bibitem{gilbert1972iterative}
{\sc P.~Gilbert}, {\em Iterative methods for the three-dimensional
  reconstruction of an object from projections}, Journal of theoretical
  biology, 36 (1972), pp.~105--117.

\bibitem{Goebel1984}
{\sc K.~Goebel and R.~Simeon}, {\em Uniform convexity, hyperbolic geometry, and
  nonexpansive mappings}, Marcel Dekker, New York and Basel, 1984.

\bibitem{gordon2010stop}
{\sc R.~Gordon}, {\em Stop breast cancer now! imagining imaging pathways toward
  search, destroy, cure, and watchful waiting of premetastasis breast cancer},
  in Breast Cancer, Springer, 2010, pp.~167--203.

\bibitem{gordon1970algebraic}
{\sc R.~Gordon, R.~Bender, and G.~T. Herman}, {\em Algebraic reconstruction
  techniques (art) for three-dimensional electron microscopy and x-ray
  photography}, Journal of theoretical Biology, 29 (1970), pp.~471--481.

\bibitem{Gregor2008}
{\sc J.~Gregor and T.~Benson}, {\em Computational analysis and improvement of
  {SIRT}}, Medical Imaging, IEEE Transactions on, 27 (2008), pp.~918--924.

\bibitem{haltmeier2009}
{\sc M.~Haltmeier}, {\em Convergence analysis of a block iterative version of
  the loping {Landweber}--{Kaczmarz} iteration}, Nonlinear Analysis: Theory,
  Methods \& Applications, 71 (2009), pp.~e2912--e2919.

\bibitem{HANSEN2018air}
{\sc P.~C. Hansen and J.~S. J{\o}rgensen}, {\em {AIR} {T}ools {II}: algebraic
  iterative reconstruction methods, improved implementation}, Numerical
  Algorithms, 79 (2018), pp.~107--137.

\bibitem{herman2009fundamentals}
{\sc G.~T. Herman}, {\em Fundamentals of {C}omputerized {T}omography: {I}mage
  {R}econstruction from {P}rojections}, Springer Science \& Business Media,
  2009.

\bibitem{Hudson1994}
{\sc H.~M. Hudson and R.~S. Larkin}, {\em Accelerated image reconstruction
  using ordered subsets of projection data}, IEEE transactions on medical
  imaging, 13 (1994), pp.~601--609.

\bibitem{JW2003}
{\sc M.~Jiang and G.~Wang}, {\em Convergence studies on iterative algorithms
  for image reconstruction}, Medical Imaging, IEEE Transactions on, 22 (2003),
  pp.~569--579.

\bibitem{kaczmarz1937}
{\sc S.~Kaczmarz}, {\em Angen{\"a}herte {Aufl{\"o}sung} von {Systemen} linearer
  {Gleichungen}}, Bulletin International de l{'}Academie Polonaise des Sciences
  et des Lettres, 35 (1937), pp.~355--357.

\bibitem{avinash2001principles}
{\sc A.~C. Kak and M.~Slaney}, {\em Principles of {C}omputerized {T}omographic
  {I}maging}, S{IAM}, 2001.

\bibitem{KL-14}
{\sc S.~Kindermann and A.~Leitao}, {\em Convergence rates for {K}aczmarz-type
  regularization methods.}, Inverse Problems \& Imaging, 8 (2014).

\bibitem{stoch2021}
{\sc V.~I. Kolobov, S.~Reich, and R.~Zalas}, {\em Finitely convergent
  deterministic and stochastic iterative methods for solving convex feasibility
  problems}, Mathematical Programming,  (2021), pp.~1--21,
  https://doi.org/10.1007/s10107--021--01628--z.

\bibitem{Mou10}
{\sc A.~Moudafi}, {\em The split common fixed-point problem for
  demi-contractive mappings}, Inverse problems, 26 (2010), p.~05507.

\bibitem{Mou11}
{\sc A.~Moudafi}, {\em A note on the split common fixed-point problem for
  quasi-nonexpansive operators}, Nonlinear Analysis: Theory, Methods \&
  Applications, 74 (2011), pp.~4083--4087.

\bibitem{MG18}
{\sc A.~Moudafi and A.~Gibali}, {\em $\ell_1$-$\ell_2$ regularization of split
  feasibility problems}, Numerical Algorithms, 78 (2018), pp.~739--757.

\bibitem{natterer1986mathematics}
{\sc F.~Natterer}, {\em The {M}athematics of {C}omputerized {T}omography},
  vol.~32, S{IAM}, 1986.

\bibitem{Ne2019}
{\sc I.~Necoara}, {\em Faster randomized block {Kaczmarz} algorithms}, SIAM
  Journal on Matrix Analysis and Applications, 40 (2019), pp.~1425--1452.

\bibitem{nedic2001incremental}
{\sc A.~Nedic and D.~P. Bertsekas}, {\em Incremental subgradient methods for
  nondifferentiable optimization}, SIAM Journal on Optimization, 12 (2001),
  pp.~109--138.

\bibitem{NP2014}
{\sc D.~Needell and J.~A. Tropp}, {\em Paved with good intentions: analysis of
  a randomized block {Kaczmarz} method}, Linear Algebra and its Applications,
  441 (2014), pp.~199--221.

\bibitem{NA2015}
{\sc T.~Nikazad, M.~Abbasi, and T.~Elfving}, {\em Error minimizing relaxation
  strategies in {Landweber} and {Kaczmarz} type iterations}, Journal of Inverse
  and Ill-posed Problems, 25 (2017), pp.~35--56.

\bibitem{NK2015}
{\sc T.~Nikazad and M.~Karimpour}, {\em Controlling noise error in block
  iterative methods}, Numerical Algorithms, 73 (2016), pp.~907--925.

\bibitem{NM2015}
{\sc T.~Nikazad and M.~Mirzapour}, {\em Projected non-stationary simultaneous
  iterative methods}, Int. J. Nonlinear Anal. Appl, 7 (2016), pp.~243--251.

\bibitem{Opi67}
{\sc Z.~Opial}, {\em Weak convergence of the sequence of successive
  approximations for nonexpansive mappings}, Bull. Amer. Math. Soc, 73 (1967),
  pp.~591--597.

\bibitem{piana1997projected}
{\sc M.~Piana and M.~Bertero}, {\em Projected {Landweber} method and
  preconditioning}, Inverse Problems, 13 (1997), p.~441.

\bibitem{Qu2008}
{\sc G.~Qu, C.~Wang, and M.~Jiang}, {\em Necessary and sufficient convergence
  conditions for algebraic image reconstruction algorithms}, IEEE Transactions
  on Image Processing, 18 (2008), pp.~435--440.

\bibitem{reich2020split}
{\sc S.~Reich, M.~T. Truong, and T.~N.~H. Mai}, {\em The split feasibility
  problem with multiple output sets in {Hilbert} spaces}, Optimization Letters,
   (2020), pp.~1--19.

\bibitem{Reich2021}
{\sc S.~Reich and T.~M. Tuyen}, {\em Projection algorithms for solving the
  split feasibility problem with multiple output sets}, Journal of Optimization
  Theory and Applications, 190 (2021), pp.~861--878.

\bibitem{RPT2020}
{\sc P.~Richt{\'a}rik and M.~Tak{\'a}c}, {\em Stochastic reformulations of
  linear systems: algorithms and convergence theory}, SIAM Journal on Matrix
  Analysis and Applications, 41 (2020), pp.~487--524.

\bibitem{saad2003}
{\sc Y.~Saad}, {\em Iterative {M}ethods for {S}parse {L}inear {S}ystems},
  S{IAM}, 2003.

\bibitem{sorensen2014multicore}
{\sc H.~H.~B. S{\o}rensen and P.~C. Hansen}, {\em Multicore performance of
  block algebraic iterative reconstruction methods}, SIAM Journal on Scientific
  Computing, 36 (2014), pp.~C524--C546.

\bibitem{Tib96}
{\sc R.~Tibshirani}, {\em Regression shrinkage and selection via the lasso},
  Journal of the Royal Statistical Society: Series B (Methodological), 58
  (1996), pp.~267--288.

\bibitem{WX11}
{\sc F.~Wang and H.-K. Xu}, {\em Cyclic algorithms for split feasibility
  problems in {H}ilbert spaces}, Nonlinear Analysis: Theory, Methods \&
  Applications, 74 (2011), pp.~4105--4111.

\bibitem{Xu10}
{\sc H.-K. Xu}, {\em Iterative methods for the split feasibility problem in
  infinite-dimensional {Hilbert} spaces}, Inverse Problems, 26 (2010),
  p.~105018.

\bibitem{yukawa2010multi}
{\sc M.~Yukawa, K.~Slavakis, and I.~Yamada}, {\em Multi-domain adaptive
  filtering by feasibility splitting}, in 2010 IEEE International Conference on
  Acoustics, Speech and Signal Processing, IEEE, 2010, pp.~3814--3817.

\bibitem{Zei90}
{\sc E.~Zeidler}, {\em Nonlinear {F}unctional {A}nalysis and its
  {A}pplications, II{B}: {N}onlinear {M}onotone {O}perators}, Springer, New
  York, 1990.

\end{thebibliography}
